\documentclass[a4paper, reqno, 14pt]{amsart}

\usepackage[usenames,dvipsnames]{color}
\usepackage{amsthm,amsfonts,amssymb,amsmath,amsxtra, array}
\usepackage[all]{xy}
\usepackage{xr-hyper}
\usepackage{dynkin-diagrams}
\usepackage{verbatim}
\usepackage{etex, tikz-cd, epic}
\usepackage[margin=1.25in]{geometry}
\usepackage{mathrsfs}
\usepackage{bbm}
\usepackage[cal=boondox]{mathalfa}
\usepackage{float}
\restylefloat{table}
\usepackage{MnSymbol}

\RequirePackage{xspace}
\RequirePackage{etoolbox}
\RequirePackage{varwidth}
\RequirePackage[shortlabels]{enumitem}
\RequirePackage{mathtools}
\RequirePackage{longtable}

\def\ge{\geqslant}
\def\le{\leqslant}
\def\a{\alpha}

\def\D{\Delta}

\def\o{\omega}

\def\s{\sigma}
\def\t{\tau}

\def\l{\lambda}

\def\i{^{-1}}
\def\ik {^{(k)}}

\def\<{\langle}
\def\>{\rangle}

\newcommand{\fkp}{\ensuremath{\mathfrak{p}}\xspace}

\newcommand{\fkO}{\ensuremath{\mathfrak{O}}\xspace}

\newcommand{\BF}{\ensuremath{\mathbb {F}}\xspace}
\newcommand{{\BG}}{\ensuremath{\mathbb {G}}\xspace}

\newcommand{{\BK}}{\ensuremath{\mathbb {K}}\xspace}

\newcommand{\BN}{\ensuremath{\mathbb {N}}\xspace}

\newcommand{\BQ}{\ensuremath{\mathbb {Q}}\xspace}
\newcommand{\BR}{\ensuremath{\mathbb {R}}\xspace}
\newcommand{\BS}{\ensuremath{\mathbb {S}}\xspace}

\newcommand{\BZ}{\ensuremath{\mathbb {Z}}\xspace}

\newcommand{\CA}{\ensuremath{\mathcal {A}}\xspace}
\newcommand{\CB}{\ensuremath{\mathcal {B}}\xspace}

\newcommand{\CF}{\ensuremath{\mathcal {F}}\xspace}

\newcommand{\CH}{\ensuremath{\mathcal {H}}\xspace}
\newcommand{\CI}{\ensuremath{\mathcal {I}}\xspace}

\newcommand{\CT}{\ensuremath{\mathcal {T}}\xspace}

\newcommand{\CX}{\ensuremath{\mathcal {X}}\xspace}
\newcommand{\CY}{\ensuremath{\mathcal {Y}}\xspace}

\newcommand{\tF}{{\widetilde{F}}}

\newcommand{\tW}{\widetilde{W}}

\newcommand{\ta}{{\widetilde{a}}}

\newcommand{\Ad}{{\mathrm{Ad}}}
\newcommand{\ad}{{\mathrm{ad}}}

\DeclareMathOperator{\diag}{diag}

\DeclareMathOperator{\Gal}{Gal}
\newcommand{\GL}{\mathrm{GL}}

\DeclareMathOperator{\Ker}{Ker}

\DeclareMathOperator{\Nm}{Nm}

\newcommand{\PGL}{{\mathrm{PGL}}}

\DeclareMathOperator{\Res}{Res}

\newcommand{\Sc}{{\mathrm{sc}}}

\newcommand{\U}{\mathrm{U}}

\def\tW{\tilde W}

\newtheorem{theorem}{Theorem}
\newtheorem{proposition}[theorem]{Proposition}
\newtheorem{lemma}[theorem]{Lemma}

\newtheorem{corollary}[theorem]{Corollary}

\theoremstyle{definition}
\newtheorem{definition}[theorem]{Definition}

\newtheorem{remark}[theorem]{Remark}

\def\af{\text{af}}
\def\fin{\text{fin}}
\def\der{\text{der}}

\def\kk{\mathbf k}

\numberwithin{equation}{section}
\numberwithin{theorem}{section}

\setitemize[0]{leftmargin=*,itemsep=\the\smallskipamount}
\setenumerate[0]{leftmargin=*,itemsep=\the\smallskipamount}

\renewcommand{\to}{%
   \ifbool{@display}{\longrightarrow}{\rightarrow}%
   }
\let\shortmapsto\mapsto
\renewcommand{\mapsto}{%
   \ifbool{@display}{\longmapsto}{\shortmapsto}%
   }
\newlength{\olen}
\newlength{\ulen}
\newlength{\xlen}
\newcommand{\xra}[2][]{%
   \ifbool{@display}%
      {\settowidth{\olen}{$\overset{#2}{\longrightarrow}$}%
       \settowidth{\ulen}{$\underset{#1}{\longrightarrow}$}%
       \settowidth{\xlen}{$\xrightarrow[#1]{#2}$}%
       \ifdimgreater{\olen}{\xlen}%
          {\underset{#1}{\overset{#2}{\longrightarrow}}}%
          {\ifdimgreater{\ulen}{\xlen}%
             {\underset{#1}{\overset{#2}{\longrightarrow}}}
             {\xrightarrow[#1]{#2}}}}%
      {\xrightarrow[#1]{#2}}
   }
\makeatother
\newcommand{\xyra}[2][]{%
   \settowidth{\xlen}{$\xrightarrow[#1]{#2}$}%
   \ifbool{@display}%
      {\settowidth{\olen}{$\overset{#2}{\longrightarrow}$}%
       \settowidth{\ulen}{$\underset{#1}{\longrightarrow}$}%
       \ifdimgreater{\olen}{\xlen}%
          {\mathrel{\xymatrix@M=.12ex@C=3.2ex{\ar[r]^-{#2}_-{#1} &}}}%
          {\ifdimgreater{\ulen}{\xlen}%
             {\mathrel{\xymatrix@M=.12ex@C=3.2ex{\ar[r]^-{#2}_-{#1} &}}}
             {\mathrel{\xymatrix@M=.12ex@C=\the\xlen{\ar[r]^-{#2}_-{#1} &}}}}}%
      {\mathrel{\xymatrix@M=.12ex@C=\the\xlen{\ar[r]^-{#2}_-{#1} &}}}%
   }
\makeatletter
\newcommand{\xla}[2][]{%
   \ifbool{@display}%
      {\settowidth{\olen}{$\overset{#2}{\longleftarrow}$}%
       \settowidth{\ulen}{$\underset{#1}{\longleftarrow}$}%
       \settowidth{\xlen}{$\xleftarrow[#1]{#2}$}%
       \ifdimgreater{\olen}{\xlen}%
          {\underset{#1}{\overset{#2}{\longleftarrow}}}%
          {\ifdimgreater{\ulen}{\xlen}%
             {\underset{#1}{\overset{#2}{\longleftarrow}}}
             {\xleftarrow[#1]{#2}}}}%
      {\xleftarrow[#1]{#2}}
   }
\newcommand{\isoarrow}{%
   \ifbool{@display}{\overset{\sim}{\longrightarrow}}{\xrightarrow\sim}%
   }
 \newcommand{\bF}{{\breve{F}}}
  
\newcommand{\bW}{{\breve{W}}}
\newcommand{\bl}{{\breve{l}}}
\newcommand{\bb}{{\breve{b}}}
\newcommand{\bI}{{\breve{I}}}

\newcommand{\bP}{{\breve{P}}}
\newcommand{\bT}{{\breve{T}}}

\newcommand{\HI}{\mathbbm{\hat 1}}

\newcommand{\ba}{{\breve{a}}}
\newcommand{\bc}{{\breve{c}}}

\newcommand{\bs}{{\breve{s}}}

\newcommand{\by}{{\breve{y}}}
\newcommand{\bz}{{\breve{z}}}
\newcommand{\bzg}{{\breve{z}_0}}

\newcommand{\bw}{{\breve{w}}}
\newcommand{\bv}{{\breve{v}}}

\newcommand{\bal}{{\breve{\mathbcal{a}}}}

\newcommand{\bBS}{\breve{\BS}}

\newcommand{\btau}{{\breve{\tau}}}
\newcommand{\bmu}{{\breve{\mu}}}
\newcommand{\bnu}{{\breve{\nu}}}
\newcommand{\be}{{\breve{\eta}}}
\newcommand{\bea}{{\be_{\ad}}}

\newcommand{\blambda}{{\breve{\lambda}}}
\newcommand{\nua}{{\bnu_{\ad}}}
\newcommand{\nuag}{{\bnu_{\ad,0}}}

\newcommand{\al}{{{\mathbcal{a}}}}

\begin{document}

\title[]{Tits groups of Iwahori-Weyl groups and presentations of Hecke algebras}
\author{Radhika Ganapathy}
\address{Department of Mathematics, Indian Institute of Science, Bengaluru, Karnataka  - 560012, India.}
\email{radhikag@iisc.ac.in}

\author[Xuhua He]{Xuhua He}
\address{The Institute of Mathematical Sciences and Department of Mathematics, The Chinese University of Hong Kong, Shatin, N.T., Hong Kong SAR, China}
\email{xuhuahe@math.cuhk.edu.hk}



\begin{abstract}
Let $G$ be a connected reductive group over a non-archimedean local field $F$ and $I$ be an Iwahori subgroup of $G(F)$. Let $I_n$ is the $n$-th Moy-Prasad filtration subgroup of $I$. The purpose of this paper is two-fold: to give some nice presentations of the Hecke algebra of connected, reductive groups with $I_n$-level structure; and to introduce the Tits group of the Iwahori-Weyl group of groups $G$ that split over an unramified extension of $F$. 

The first main result of this paper is a presentation of the Hecke algebra $\CH(G(F),I_n)$, generalizing the previous work of Iwahori-Matsumoto on the affine Hecke algebras. For split $GL_n$, Howe gave a refined presentation of the Hecke algebra $\CH(G(F),I_n)$. To generalize such a refined presentation to other groups requires the existence of some nice lifting of the Iwahori-Weyl group $W$ to  $G(F)$. The study of a certain nice lifting of $W$ is the second main motivation of this paper, which we discuss below. 

In 1966, Tits introduced a certain subgroup of $G(\kk)$, which is an extension of $W$ by an elementary abelian $2$-group. This group is called the Tits group and provides a nice lifting of the elements in the finite Weyl group. The ``Tits group'' $\CT$ for the Iwahori-Weyl group $W$ is a certain subgroup of $G(F)$, which is an extension of the Iwahori-Weyl group $W$ by an elementary abelian $2$-group. The second main result of this paper is a construction of Tits group $\CT$ for $W$ when $G$ splits over an unramified extension of $F$.  As a consequence, we generalize Howe's presentation to such groups. We also show that when $G$ is ramified over $F$, such a group $\CT$ of $W$ may not exist. 
\end{abstract}

\maketitle

\bigskip

\section{Introduction}

\subsection{Presentations of Hecke algebras}

Let $G$ be a connected reductive group over a non-archimedean local field $F$. Let $I$ be an Iwahori subgroup of $G(F)$ and $W$ be the Iwahori-Weyl group of $G(F)$. Then $G(F)=\sqcup_{w \in W} I w I$. The group $W$ is a quasi-Coxeter group, namely, it is a semidirect product of an affine Weyl group $W_{\af}$ with a group $\Omega$ of length-zero elements. The Iwahori-Hecke algebra $\CH_0=\CH(G(F), I)$ is the $\BZ$-algebra of the compactly supported, $I$-biinvariant functions on $G(F)$. The Iwahori-Matsumoto presentation of $\CH_0$ reflects the quasi-Coxeter group structure of $W$: the generators of $\CH_0$ are the characteristic functions $\mathbbm{1}_{I w I}$, where $w$ runs over elements in $W$, and the relations are given by multiplications of the characteristic functions via the condition on the length functions of $W$. See Theorem \ref{prop:aff} for the precise statement.

The representations of $G(F)$ which are generated by the Iwahori-fixed vectors gives to the representations of the Iwahori-Hecke algebra $\CH_0$. Let $n \in \BN$ and $I_n$ be the $n$-th congruence subgroup of $I$. Let $\CH_n=\CH(G(F), I_n)$ be the $\BZ$-algebra of the compactly supported, $I_n$-biinvariant functions on $G(F)$. It plays a role in the study of representations of $G(F)$ with deeper level structure. 

One main purpose of this paper is to establish some nice presentations of $\CH_n$. The first main result is the generalization of the Iwahori-Matsumoto presentation to $\CH_n$: the generators are the characteristic functions on the $I_n$-double cosets on $G(F)$ and the multiplications of the characteristic functions are given via the conditions on the length function of $W$. We refer to Theorem \ref{mult} for the precise statement. As a consequence, we show that the algebra $\CH_n$ is finitely generated. 

In \cite{How85}, Howe discovered a nice presentation of $\CH_n$ when $G=GL_n(F)$. Here the generators are the characteristic functions $\mathbbm{1}_{g I_n}$ for $g \in I/I_n$ and $\mathbbm{1}_{I_n m(w) I_n}$, where $w$ runs over elements of $W$ of length $0$ and $1$, and $m(w)$ is a nice representative of $w$ in $G(F)$. This presentation is a refinement of the Iwahori-Matsumoto presentation and has some nice applications to the representation theory of $p$-adic groups. Howe's presentation was later generalized by the first-named author to split groups. We observed that such refined presentation requires the existence of the nice lifting of the Iwahori-Weyl group $W$ to $G(F)$. Such a lifting, which we introduce in \S\ref{tits}, is motivated by Tits work on finite Weyl groups. We call such a lifting the {\it Tits group} of the Iwahori-Weyl group $W$ and call the refined presentation of $\CH_n$ the Howe-Tits presentation. In Theorem \ref{thm:howe-tits}, we show that if the Tits group for $W$ exists, then the algebra $\CH_n$ admits the Howe-Tits presentation. 

\subsection{Tits groups of the finite Weyl groups and Iwahori-Weyl groups}\label{sec:intro-1} Now we come to the second main purpose of this paper: the study of the Tits groups. 

We first make a short digression and discuss Tits groups of finite Weyl groups. Let $G$ be a connected reductive group split over a field $\mathfrak{F}$ and $W_0$ be its absolute Weyl group. Tits in \cite{Ti} introduced the Tits group $\CT$ of $W_0$. It is a subgroup of $G(\mathfrak{F})$, which is an extension of $W_0$ by $T_2$, where $T_2$ is the elementary abelian subgroup generated by $\a^\vee(-1)$, where $\a$ runs over all the roots in $G$. Moreover, for any $w \in W_0$, there exists a nice lifting $n_w \in \CT$. These liftings have nice properties: 

\begin{enumerate}
   
    \item $n_{s_\a}^2=\a^\vee(-1)$ for any simple root $\a$. 
    
    \item The set $\{n_s\}$ for simple reflections $s$ satisfies the Coxeter relations, i.e., for any simple reflections $s$ and $s'$, we have $$n_{s} n_{s'} \cdots=n_{s'} n_s \cdots,$$ where each side of the expression above has $k(s, s')$ factors. Here $k(s, s')$ is the order of $s s'$.
\end{enumerate}

We refer to the recent work of Reeder, Levy, Yu and Gross \cite{RLYG}, Adams and the second-named author \cite{AH17} and Rostami \cite{Ro} for some further study of the elements $n_w$ and its applications to supercuspidal representations of $p$-adic groups. 

Now let us come back to the group $G(F)$. Our second main result of this paper is the construction of a Tits group $\CT$ of the Iwahori-Weyl group of a connected, reductive group $G$ that is $\bF$-split. We establish in Theorem \ref{thm:tits-F} that 

\begin{theorem}
We have the short exact sequence $$1 \to S_2 \to \CT \to W \to 1.$$

Moreover, for any $w \in W$, there exists a lifting $n_w \in \CT$ such that 

\begin{itemize}
    \item For any affine simple reflection $s_a$, $n_{s_a}^2=b^\vee(-1)$ where $b$ is the gradient of $a$.

    \item We have $n_w=n_{s_{i_1}} n_{s_{i_2}} \cdots n_{s_{i_n}} n_\t$ for any reduced expression $w=s_{i_1} s_{i_2} \cdots s_{i_n} n_\t$, where $\t \in \Omega$ and $s_{i_1}, \ldots, s_{i_n}$ are simple reflections. 
\end{itemize}
\end{theorem}

We refer to \S \ref{sec:TitsF} for the definition of the elementary abelian $2$-group $S_2$. 

As a consequence, we have the Howe-Tits presentation of $\CH_n$ for groups that are $\bF$-split.

It is also worth pointing out that for ramified groups, such a $\CT$ may not exist. We give an example in \S \ref{sec:example}. 

\subsection{The difficulty and strategy}  In this subsection, we describe the strategy that goes into the construction of the Tits group of the Iwahori-Weyl group $W$ of $G$ over $F$. 

The Tits group of the finite absolute Weyl group is constructed via a ``pinning'' of $G(\mathfrak{F})$. Roughly speaking, a pinning gives a collection of isomorphisms from additive group $\BG_a$ to the simple root subgroups of $G$. Given a pinning, one may define the lifting of simple reflections $n_s$ and check that the conditions (1) \& (2) in \S\ref{sec:intro-1} are satisfied. The Tits group of the finite Weyl group is generated by the $n_s$ where $s$ varies over the finite simple reflections.

When $G_F$ is not quasi-split, the group need not admit a ``pinning" analogous to the one discussed above, and hence there is no natural choice of representatives for the elements of the relative or affine Weyl group over $F$. 

We construct the Tits group of the Iwahori-Weyl group of $G$ over $F$ in two steps.  We first construct the Tits group of Iwahori-Weyl group over $\bF$, where  $\bF$ is the completion of the maximal unramified extension of $F$ contained in a fixed separable closure of $F$. Next, we ``descend" this construction down to $F$. The advantage of this approach is that the group $G_\bF$ is always quasi-split and admits a nice system of pinnings analogous to the one discussed in the preceding paragraph. 
 
We now explain these two steps in more detail. 
\begin{enumerate}
\item Let $G$ be a connected, reductive group over $F$ such that $G_\bF$ is $\bF$-split and let $T$ be a maximal $F$-torus in $G$ that is $\bF$-split. Let $\bal$ be a $\sigma$-stable alcove in the apartment $\CA(T, \bF)$ and let $\bBS$ be the set of affine simple reflections through the walls of $\bal$. To choose representatives of the elements of $\bBS$, we introduce an affine pinning; for each affine simple root $\ba$ with gradient $\bb$, this is a homomorphism $x_\ba: \BG_a \rightarrow U_\bb$ such that the image of $m(s_{\ba}) = x_\ba(1) x_{-\ba}(1) x_\ba(1)$ in the affine Weyl group is $s_\ba$. We then show that this set of representatives satisfy Coxeter relations and furthermore,  $m(s_{\ba})^2 = \bb^\vee(-1)$ for each affine simple reflection $\ba$. We show that the group generated by  $\{\breve\lambda(\varpi_F)\;|\; \breve\lambda \in X_*(T)\}$, the $\{m(\bs)\;|\; \bs \in \bBS\}$, and the group $\breve S_2 = \langle \bb^\vee(-1)\;|\; \bb \in \Phi(G,T)\rangle$ yields a Tits group of the Iwahori-Weyl group over $\bF$. We also include an example here of a wildly ramified unitary group over $\bF$ for which the Tits group of the Iwahori-Weyl group over $\bF$ does not exist. This is done in \S \ref{sec:TGbF}. 
\item We now explain the descent step. Let $\sigma$ denote the Frobenius morphism on $G_\bF$ such that the $F$-structure it yields is $G$. Let $\al = \bal^\sigma$ and let $\BS$ be the set of reflections through the walls of $\al$. Then $\BS$ generates the Coxeter group $W_{\af}$ and $W = W_\af \rtimes \Omega_{\al}$ where $\Omega_\al$ is the stabilizer of the alcove $\al$. By the work of Lusztig \cite{Lusz} it is known that the elements of $\BS$ correspond to certain ``nice" $\sigma$-orbits in $\bBS$.  We construct an affine pinning over $\bF$ such that the set of representatives $\{m(\bs)\;|\; \bs \in \CX\}$ obtained using this pinning is $\sigma$-stable for each of these nice $\sigma$-orbits $\CX$. This descent argument yields a set of representatives in $G(F)$ for the elements of $\BS$ that satisfy Coxeter relations. This is done in \S \ref{HDSS}. 

Let $\breve \CT$ be as in (1), but with the representatives $\{m(\bs)\;|\; \bs \in \bBS\}$ as in the preceding paragraph. Then $\breve\CT$ is $\sigma$-stable. We need to show that $\breve\CT^{\sigma} \subset G(F)$ is a Tits group of $W$ over $F$. The most difficult part of the argument is to carry out the descent step for the elements of $\Omega_{\bal}$; since $\Omega_\al = \Omega_{\bal}^\sigma$, we need to show that for $\breve\tau \in \Omega_\bal^\sigma$, there is a representative $m(\btau)$ of $\btau$ in $\breve\CT$ with $\sigma(m(\btau)) = m(\btau)$. The results from our previous construction allow us to choose a representative $m(\btau)$ of $\btau$ in $\breve\CT$ with the property that $\sigma(m(\btau)) = cm(\btau)$ for a suitable $c \in \breve S_2$, and a priori, we do not have any control over this element $c \in \breve S_2$. 
We carry out this step by constructing a Frobenius morphism $\sigma$ associated to the $F$-isomorphism class of $G$, and then construct the Tits group $\breve\CT \subset G(\bF)$ of $\bW$ over $\bF$ so that $\breve\CT^\sigma \subset G(\bF)^\sigma$ is a Tits group of $W$ over $F$.  This is done in \S\ref{FMIF} - \S\ref{constaubs} and some parts of the argument are based on a case-by-case analysis. 
\end{enumerate}

\smallskip

\subsection*{Acknowledgments:} The authors would like to thank T. Haines, G. Lusztig and M.-F. Vigneras for useful discussions. R.G.~ would like to thank the Infosys foundation for their support through the Young Investigator award. X.H.~is partially supported by a start-up grant and by funds connected with Choh-Ming Chair at CUHK, and by Hong Kong RGC grant 14300220.

\section{Preliminaries}

\subsection{Notation}\label{notation} Let $F$ be a non-archimedean local field with $\fkO_F$ its ring of integers, $\fkp_F$ its maximal ideal, $\varpi_F$ a uniformizer, and $\kk=\BF_q$ its residue field. Let $p$ be the characteristic of $\kk$. Let $\bar F$ be the completion of a separable closure of $F$. Let $\bF$ be the completion of the maximal unramified subextension with valuation ring $\fkO_\bF$ and residue field $\bar\kk$. Note that $\varpi_F$ is also a uniformizer of $\bF$. Let $\Gamma = \Gal(\bar F/F)$ and $\Gamma_0=\Gal(\bar F/\bF)$. 

Let $G$ be a connected, reductive group over $F$. By Steinberg's Theorem (see \cite[Theorem 56]{Ste65}), $G_\bF$ is quasi-split. Let $\sigma$ denote the Frobenius action on $G(\bF)$ such that $G(F) = G(\bF)^\sigma$. Let $A$ be a maximal $F$-split torus of $G$ and $S$ be a maximal $\bF$-split $F$-torus of $G$ containing $A$. Let $T = Z_G(S)$. Then $T$ is defined over $F$ and is a maximal $F$-torus of $G$ containing $S$. Let $\tF$ be the field of invariants of the kernel of the representation of $\Gamma_0$ on $X^*(T)$. This extension is Galois over $\bF$. Hence $T$ and $G$ are split over $\tF$. By \cite[Chapter V, \S4, Proposition 7]{Ser79}, there exists a uniformizer $\varpi_\tF$ of $\tF$ with $\Nm_{\tF/\bF}(\varpi_\tF) = \varpi_F$, where $\Nm_{\tF/\bF}$ is the norm map. Fix one such.  

Let $\tilde\Phi(G,T)$ be the set of roots of $T_\tF$ in $G_\tF$. Then the set of relative roots of $S$ in $G_\bF$, denoted by $\breve\Phi(G,S)$, is the set of the restrictions of the elements in $\tilde\Phi(G,T)$ to $S$. Let $\bW_0$ denote the relative Weyl group of $G$ with respect to $S$ and let $W(G,T)$ denote the absolute Weyl group of $G$.

Let $\CB(G, \bF)$ (resp. $\CB(G,F)$) denote the enlarged Bruhat-Tits building of $G(\bF)$ (resp. $G(F)$). Then $\CB(G, \bF)$ carries an action of $\sigma$ and $\CB(G,F)= \CB(G, \bF)^\sigma$. Let $\CA(S, \bF)$ be the apartment in $\CB(G, \bF)$ corresponding to $S$. Let $\bal$ be a $\sigma$-stable alcove in $\CA(S,\bF)$. Let $\bv_0$ be a special vertex contained in the closure of $\bal$. Set $\al = \bal^\sigma$; this is an alcove in the apartment $\CA(A, F)$ (see \cite[\S 5.1]{BT2}).

Let $\breve\Phi_\af(G,S)$ denote the set of affine roots of $G(\bF)$ relative to $S$. Let $V = X_*(S) \otimes_\BZ \BR$. The choice of $\bv_0$ also allows us to identify $\CA(S, \bF)$ with $V$ via $\bv_0 \mapsto 0 \in V$, which we now do. We then view $\bal \subset V$. Let $\breve\Delta \subset \breve\Phi_\af(G,S)$ be the set of affine roots such that the corresponding vanishing hyperplanes form the walls of $\bal$. The Weyl chamber in $V$ that contains $\bal$ then yields a set of simple roots for $\breve\Phi(G,S)$ which we denote as $\breve\Delta_0$. Clearly $\breve\Delta_0 \subset \breve\Delta$.

\subsection{Iwahori-Weyl group over $\breve F$}

Let $\bI$ be the Iwahori subgroup associated to $\bal$. Let $\kappa_{T, \bF}: T(\bF) \rightarrow X_*(T)_{\Gamma_0}$ denote the Kottwitz homomorphism. The map $\kappa_{T, \bF}$ is surjective  and its kernel $T(\bF)_1$ is the unique parahoric subgroup of $T(\bF)$. By \cite[\S 7.2]{Kot97}, we have the following commutative diagram
\begin{equation}\label{KottwitzTorus}
\begin{tikzcd}
T(\tF) \arrow{r}{\kappa_{T,\tF}} \arrow{d}{\Nm_{\tF/\bF}}
&X_*(T) \arrow{d}{pr}\\
 T(\bF)\arrow{r}{\kappa_{T, \bF}} &X_*(T)_{\Gamma_0}.
\end{tikzcd}\end{equation}
Let $\bW = N_G(S)(\bF)/T(\bF)_1$ be the Iwahori-Weyl group of $G(\bF)$ with length function $\bl$. This group fits into an exact sequence
\[1 \rightarrow X_*(T)_{\Gamma_0} \rightarrow \tW \rightarrow \bW_0 \rightarrow 1.\]

Recall that we have chosen a special vertex $\bv_0$. With this, we have a semi-direct product decomposition
\begin{align}\label{decomp1}
    \bW\cong X_*(T)_{\Gamma_0} \rtimes \bW_0.
\end{align}

Let $\bBS = \{s_\ba\;|\; \ba \in \breve\Delta\}$ be the set of simple reflections with respect to the walls of $\bal$. Let $\bBS_0 =\{s_\ba\;|\; \ba \in \breve\Delta_0\}$. Let $\bW_{\af} \subset \bW$ be the Coxeter group generated by $\bBS$. Let $T_{\Sc}, N_{\Sc}$ denote the inverse images of $T\cap G_\der$, resp. $N_
G(S)\cap G_\der$ in $G_\Sc$. Let $S_{\Sc}$ denote the split component of $T_{\Sc}$. Then $\bW_{\af}$ may be identified with the Iwahori-Weyl group of $G_{\Sc}$. It fits into the exact sequence
\begin{align}\label{IWGSDP}
    1 \rightarrow \bW_{\af} \rightarrow \bW \rightarrow X^*(Z(\hat G)^{\Gamma_0}) \rightarrow 1.
\end{align}
Let $\Omega_{\bal}$ be the stabilizer of $\bal$ in $\bW$. Then $\Omega_{\bal}$ maps isomorphically to $X^*(Z(\hat G)^{\Gamma_0})$ and we have a $\sigma$-equivariant semi-direct product decomposition
\[\bW \cong \bW_{\af} \rtimes \Omega_\bal.\]

Let $\bl$ be the length function on $\bW$. Then $\bl(s) = 1$ for all $s \in \bBS$ and $\Omega_\bal$ is the set of elements of length 0 in $\bW$.

\subsection{Iwahori-Weyl group over $F$}\label{sec:WF}
Let $I$ be the Iwahori subgroup of $G(F)$ associated to $\al$. Then $I = \bI^\sigma$. Let $M = Z_G(A)$ and $M(F)_1$ be the unique parahoric subgroup of $M(F)$. We may identify $M(F)_1$  with the kernel of the Kottwitz homomorphism $M(F) \rightarrow X^*(Z(\hat M)^{\Gamma_0})^\sigma$. Let $W= N_G(A)(F)/M(F)_1$ denote the Iwahori-Weyl group of $G(F)$ with length function $l$. 

By \cite[Lemma 1.6]{Ri}, we have a natural isomorphism $W \cong \breve W \,^{\s}$. It is proved in \cite[Proposition 1.11 \& sublemma 1.12]{Ri} that 

(a) for $w, w' \in W$, $\breve \ell(w w')=\breve \ell(w)+\breve \ell(w')$ if and only if $\ell(w w')=\ell(w)+\ell(w')$.

The semi-direct product decomposition of $\bW$ in \eqref{IWGSDP} is $\sigma$-equivariant and yields a decomposition
\[W \cong \bW_{\af}^\sigma \rtimes \Omega_{\bal}^\sigma.\]
Let $W_{\af} = W_{\af}^\sigma$ and let $\BS$ be the set of reflections through the walls of $\al$. Then $(W_\af, \BS)$ is a Coxeter system. The group $\Omega_\al$, which is the stabilizer of the alcove $\al$, is isomorphic to $\Omega_{\bal}^\sigma$ and is the set of length 0 elements is $W$. 

The simple reflections $\BS$ of $W_{\af}$ are certain elements in $\bW_{\af}$. The explicit description is as follows. For any $\s$-orbit $\CX$ of $\bBS$, we denote by $\bW_{\CX}$ the parabolic subgroup of $\bW_{\af}$ generated by the simple reflections in $\CX$. If moreover, $\bW_{\CX}$ is finite, we denote by $\bw_{\CX}$ the longest element in $\bW_{\CX}$. It is proved by Lusztig \cite[Theorem A.8]{Lusz} that there exists a natural bijection $s \mapsto \CX$ from $\BS$ to the set of $\s$-orbits of $\bBS$ with $\bW_{\CX}$ finite such that the element $s \in W_{\af} \subset \bW_{\af}$ equals to $\bw_{\CX}$. 

\subsection{Moy-Prasad filtration subgroups} Let $\bI$ be the Iwahori subgroup of $G(\bF)$ associated to the alcove $\bal$. Recall that we have chosen a special point $\bv_0$ in $\CA(S, \bF)$, using which we have identified $\CA(S, \bF)$ with $V$. Let $(\phi_\ba)_{\ba \in \breve\Phi(G,S)}$ be the corresponding valuation of root datum of $(T, (U_\ba)_{\ba \in \breve\Phi(G,S)})$ (see \cite[\S 6.2]{BT1}). For $\bv \in \bal$, $\ba \in \breve\Phi(G,S)$ and $r \in \BR$, let $U_{\ba}(\bF)_{\bv, r}$ denote the filtration of the root subgroup $U_a$ (see \cite[\S 4.3 - \S 4.6]{BT2}). More precisely,
\[U_{\ba}(\bF)_{\bv, r} = \{u \in U_\ba(\bF)\;|\; \langle \ba, \bv \rangle +\phi_\ba(u) \geq r\}.\]
The subgroup $U_{\ba}(\bF)_{ \bv, 0}$ does not depend on the choice of $\bv \in \bal$ and we may denote it as $ U_{\ba}(\bF)_{ \bal, 0}$. Note that $\bI$ is generated by $T(\bF)_1$ and $U_{\ba}(\bF)_{ \bal, 0}, \ba \in \breve\Phi(G,S)$.

Let $\bI_n$ be the $n$-th Moy-Prasad filtration subgroup of $\bI$. In particular, for $n \ge 1$, $ \bI_n$ is a normal subgroup of $\bI$.  Let $\CT^{NR}$ denote the Neron-Raynaud model of $T$, a group scheme of finite type over $\fkO_F$ with connected geometric fibers such that $\CT^{NR}(\fkO_\bF) = T(\bF)_1$. Let $\breve T_n = \Ker(\CT^{NR}(\fkO_\bF) \rightarrow \CT^{NR}(\fkO_\bF/\fkp_\bF^n))$. Then $\bI_n$ is generated by $\breve T_n$ and $U_{\ba}(\bF)_{\bv_\bal, n}, \ba \in \breve \Phi(G,S)$, where $\bv_\bal$ is the barycenter of $\bal$. 

Let $I_n=\breve I_n ^\s$. Then for $n \ge 1$, $I_n$ is a normal subgroup of $I$.

\subsection{The subgroup $\bP_s$}\label{sec:Ps} Let $s \in \BS$. Let $\CX$ be the $\sigma$-stable orbit in $\bBS$ corresponding to $s$ (see \S \ref{sec:WF}). Let $\breve P_s=\sqcup_{\bw \in \bW_\CX} \breve I \dot \bw \breve I \supset \breve I$ be the parahoric subgroup of $G(\bF)$ associated to $\CX$. This is the parahoric subgroup attached to $\bal_s = \overline{\bal}^{\breve W_\CX}$, where $\overline{\bal}$ is the closure of the alcove $\bal$. Then $\bP_s$ is generated by $T(\bF)_1$ and $U_{\ba}(\bF)_{\bal_s, 0},\; \ba \in \breve \Phi(G,S)$. Let $\bP_{s,n}$ be the $n$-th Moy-Prasad filtration subgroup of $\bP_s$. It is generated by $\bT_n$ and $U_{\ba}(\bF)_{\bv_s, n}, \ba \in \breve \Phi(G,S)$ where  $\bv_{s}$ is the barycenter of $\bal_s$.

\subsection{The Hecke algebra $\CH(G(F), I_n)$}Let $\CH_n = \CH(G(F), I_n)$ be the Hecke algebra of compactly supported, $I_n$-biinvariant $\BZ$-valued functions on $G(F)$. Note that $I_0=I$. The algebra $\CH_0$ is the Iwahori-Hecke algebra.

\section{A Tits group associated to an Iwahori-Weyl group}\label{tits}

\subsection{Tits group associated to an  absolute Weyl groups} In this subsection, we assume that $\mathfrak{F}$ is any field and $G$ is a reductive group split over $\mathfrak{F}$. Let $T$ denote a maximal $F$-split torus in $G$. 

We follow \cite{Ti}. For any root $a$, we denote by $a^\vee$ the corresponding coroot. Let $S_2$ be the elementary abelian two-group generated by $\{a^\vee(-1)\}$ for all roots $a$. Associated to any pinning of $G$, we have the Tits group $\CT_{\text{fin}}$. This is a subgroup of $N_G(T)$, generated by $\{n_s\}$, where $s$ runs over the simple reflections in the absolute Weyl group $W(G, T)$ and $n_s$ is a certain lift of $s$ to $N_G(T)$. 

Below are some properties on the Tits group $\CT_{\text{fin}}$:

\begin{enumerate}
   
    \item $n_{s_a}^2=a^\vee(-1)$ for any simple root $a$. 
    
    \item The set $\{n_s\}$ for simple reflections $s$ satisfies the Coxeter relations, i.e., for any simple reflections $s$ and $s'$, we have $$n_{s} n_{s'} \cdots=n_{s'} n_s \cdots,$$ where each side of the expression above has $k(s, s')$ factors. Here $k(s, s')$ is the order of $s s'$ in $W(G, T)$.
    
     \item The map $n_s \mapsto s$ induces a short exact sequence $$1 \to T_2 \to \CT_{\text{fin}} \to W(G, T) \to 1.$$
    
\end{enumerate}

For any $w \in W(G, T)$, we may define $n_w=n_{s_1} \cdots n_{s_k} \in \CT_{\text{fin}}$, where $s_1 \cdots s_k$ is a reduced expression of $w$. As a consequence of (2), the definition of $n_w$ is independent of the choice of the reduced expression of $w$. We call the liftings $\{n_w\}_{w \in W(G, T)}$ a {\it Tits cross-section} of $W(G, T)$ in $\CT_{\fin}$. 

\subsection{A Tits group of Iwahori-Weyl group over $\breve F$}\label{sec:tits} Motivated by the construction of the Tits group of the absolute Weyl group, we introduce the Tits groups of Iwahori-Weyl groups. 

For each $\bb$ in the relative root system $\breve\Phi(G,S)$, we set 
\[\bb_* =\begin{cases} \bb, & \text{ if } \bb \text{ is reduced}; \\ \bb/2, & \text{  otherwise.} \end{cases}
\]

Note that any element $\bw \in \bW$ can be written as $\bw=\bs_{i_1} \cdots \bs_{i_n} \breve \t$, where $\bs_{i_1}, \cdots, \bs_{i_n} \in \breve \BS$ and $\breve \t \in \Omega_{\bal}$. If $n=\breve \ell(\bw)$, then we say that $\bw=\bs_{i_1} \cdots \bs_{i_n} \breve \t$ is a reduced expression of $\bw$ in $\bW$.

\begin{definition}\label{def:tits}
Let $\breve S_2$ be the elementary abelian two-group generated by $\breve b^\vee(-1)$ for $\breve b \in \breve \Phi(G, S)$. A {\it Tits group} of $\bW$ is a subgroup $\breve \CT$ of $N_G(S)(\breve F)$ such that 
\begin{enumerate}
    \item The natural projection $\breve \phi: N_G(S)(\breve F) \to \bW$ induces a short exact sequence $$1 \to \breve S_2 \to \breve \CT \xrightarrow{\breve \phi} \bW \to 1.$$
    
    \item There exists a {\it Tits cross-section} $\{m(\bw)\}_{\bw \in \bW}$ of $\bW$ in $\breve \CT$ such that 
    
    \begin{enumerate}
        \item for $\ba \in \breve\Delta$, $m(\bs_{\ba})^2 = \bb_*^\vee(-1)$, where $\bb$ is the gradient of $\ba$.
        
        \item for any reduced expression $\bw=\bs_{i_1} \cdots \bs_{i_n} \breve \t$ in $\bW$, we have $m(\bw)=m(\bs_{i_1}) \cdots m(\bs_{i_n}) m(\breve \t)$. 
    \end{enumerate}
\end{enumerate}
\end{definition}

It is easy to see that the condition (2) (b) in Definition \ref{def:tits} is equivalent to 

Condition (2)(b)$^{\dagger}$: $m(\bw \bw')=m(\bw) m(\bw')$ for any $\bw \in \bW_{\af}$ and $\bw' \in \bW$ with $\breve \ell(\bw \bw')=\breve \ell(\bw)+\breve \ell(\bw')$.

Suppose that a Tits group $\breve \CT$ of $\bW$ exists and $\breve \phi: \breve \CT \to \bW$ is the projection map. Let $\breve \CT_{\af}=\breve \phi \i(\breve W_{\af})$. This is the subgroup of $\breve \CT$ generated by $\breve S_2$ and $m(\bw)$ for $w \in \bW_{\af}$. We have the following commutative diagram
\[
\begin{tikzcd}
  1 \arrow[r] & \breve S_2 \arrow[d, equal] \arrow[r] & \breve\CT_{\af} \arrow[d, hook] \arrow[r,"\breve\phi"] & \bW_{\af} \arrow[d, hook] \arrow[r] & 1 \\
  1 \arrow[r] & \breve S_2 \arrow[r] & \breve\CT \arrow[r,"\breve\phi"] & \bW \ar[r] & 1.
\end{tikzcd}
\]

For $\btau \in \Omega_\bal$, any lifting $m'(\btau)$ of $\btau$ in $G(\bF)$ lies in the normalizer of $\bI$, where $\bI$ is the Iwahori subgroup attached to the alcove $\bal$. This a special case of the fact that for $g \in G(\bF)$ and $x \in \CA(S, \bF)$, $g \bP_x g\i = \bP_{g \cdot x}$, where $\bP_x$ is the parahoric subgroup attached to $x$. 

Let us add some comments on $\breve b_*$. In this paper, we will construct the Tits group for connected reductive groups split over $\breve F$. For these groups, there is no difference between $\breve b_*$ and $\breve b$. We expect that Tits groups (in the Definition \ref{def:tits}) exist for tamely ramified groups. Then one needs to use $\breve b_*$ instead of $\breve b$ in condition (2) (a), e.g., for the tamely ramified unitary groups. 

\subsection{Tits groups over $F$}\label{sec:TitsF}
The Tits group of $W$ is defined as follows.  

Note that any element $w \in W$ can be written as $w=s_{i_1} \cdots s_{i_n} \t$, where $s_{i_1}, \cdots, s_{i_n} \in \BS$ and $\t \in \Omega_{\al}$. If $n=\ell(w)$, then we say that $w=s_{i_1} \cdots s_{i_n} \t$ is a reduced expression of $w$ in $W$. 

\begin{definition}\label{def:tits2}
Let $S_2=\breve S_2^\s$. A {\it Tits group} of $W$ is a subgroup $\CT$ of $N_G(A)(F)$ such that 
\begin{enumerate}
    \item The natural projection $\phi: N_G(A)(F) \to W$ induces a short exact sequence $$1 \to S_2 \to \CT \xrightarrow{\phi} W \to 1.$$
    
    \item There exists a {\it Tits cross-section} $\{m(w)\}_{w \in W}$ of $W$ in $\CT$ such that 
    
    \begin{enumerate}
       \item for $a \in \Delta$, $m(s_{a})^2 = b^\vee(-1)$, where $b$ is the gradient of $a$.
        
       \item for any reduced expression $w=s_{i_1} \cdots s_{i_n} \t$ in $W$, we have $m(w)=m(s_{i_1}) \cdots m(s_{i_n}) m(\t)$. 
    \end{enumerate}
\end{enumerate}
\end{definition}

Note that in general, $S_2$ is larger than than the subgroup generated by $b^\vee(-1)$ for $b \in \Phi(G, A)$. 

Suppose that a Tits group $\CT$ of $W$ exists and $\phi: \CT \to W$ is the projection map. Let $\CT_{\af}=\phi \i(W_{\af})$. This is the subgroup of $\CT$ generated by $S_2$ and $m(w)$ for $w \in W$. We have the following commutative diagram
\[
\begin{tikzcd}
  1 \arrow[r] & \breve S_2^\s \arrow[d, equal] \arrow[r] & \CT_{\af} \arrow[d, hook] \arrow[r,"\phi"] & W_{\af} \arrow[d, hook] \arrow[r] & 1 \\
  1 \arrow[r] & \breve S_2^\s \arrow[r] & \CT \arrow[r,"\phi"] & W \ar[r] & 1.
\end{tikzcd}
\]

We would like to point out that unlike Tits groups of absolute Weyl groups for split reductive groups, the Tits groups of $\bW$ and $W$ may not exist in general. See \S \ref{sec:example}. 

\section{Two presentations of the Hecke algebra $\CH(G(F), I_n)$}\label{Bruhat}

We first recall the Iwahori-Matsumoto presentation of the Iwahori-Hecke algebras:

\begin{theorem}\label{prop:aff}
For $w \in W$, let $\dot w$ be any representative of $w$ in $G(\bF)$. The Hecke algebra $\CH_0$ is a free module with basis $\{\mathbbm{1}_{I \dot w I}\}_{w \in W}$ and the multiplication is given by the following formulas:

(1) $\mathbbm{1}_{I \dot w I} \mathbbm{1}_{I \dot w' I}=\mathbbm{1}_{I \dot w \dot w' I}$ if $\ell(w w')=\ell(w)+\ell(w')$.

(2) $\mathbbm{1}_{I \dot s I} \mathbbm{1}_{I \dot w I}=(q^{\breve \ell(s)}-1) \mathbbm{1}_{I \dot w I}+q^{\breve \ell(s)} \mathbbm{1}_{I \dot s \dot w I}$ for $s \in \BS$ and $w \in W$ with $s w<w$. 
\end{theorem}

The first main result of this section is the following similar presentation for $\CH_n$ for $n \ge 1$. We call it the Iwahori-Matsumoto presentation for $\CH_n$. 
\begin{theorem}\label{mult} 
Let $n \ge 1$. The algebra $\CH_n$ is generated by $\mathbbm{1}_{I_n g I_n}$ for $g \in G(F)$ subject to the following relations:

(0) If $g$ and $g'$ are in the same $I_n \times I_n$-coset of $G(F)$, then $\mathbbm{1}_{I_n g I_n}=\mathbbm{1}_{I_n g' I_n}$.

(1) If $\ell(\pi(g g'))=\ell(\pi(g))+\ell(\pi(g'))$, then $$\mathbbm{1}_{I_n g I_n} *\mathbbm{1}_{I_n g' I_n}=\mathbbm{1}_{I_n g g' I_n}.$$ 

(2) If $\pi(g)=s \in \BS$, then $$\mathbbm{1}_{I_n g I_n} *\mathbbm{1}_{I_n g' I_n}=\begin{cases} q^{\breve \ell(s)} \mathbbm{1}_{I_n g g' I_n}, & \text{ if } \pi(g g')=\pi(g'), \\ q^{\breve \ell(s)} \sum_{u \in P_{s, n}/I_n} \mathbbm{1}_{I_n u g g' I_n}, & \text{ if } \pi(g g')<\pi(g').\end{cases}$$
\end{theorem}

In the above, $P_{s,n} = \bP_{s, n}^\sigma$ with $\bP_{s,n}$ as in \S \ref{sec:Ps}.  Note that if $s w<w$, then $(I \dot s I) (I \dot w I)=I \dot w I \sqcup I \dot s \dot w I$. For any $g \in I \dot s I$ and $g' \in I \dot w I$, there are two possibilities: either $(I_n g I_n) (I_n g' I_n) \subset I \dot w I$ or $(I_n g I_n) (I_n g' I_n) \subset I \dot s \dot w I$. Thus there are two cases for the multiplication $\mathbbm{1}_{I_n g I_n} * \mathbbm{1}_{I_n g' I_n}$. 

\subsection{Collection of some results from \cite{He-1}}\label{he1}
We define the map $$\pi: G(F) \to W, \quad g \mapsto w \text{ for } g \in I \dot w I.$$ It is proved in \cite[Lemma 4.5 \& Lemma 4.6]{He-1} that for any $g \in G(F)$, $(\breve I_n g \breve I_n/\breve I_n)^\s=I_n g I_n/I_n$ and $\sharp I_n g I_n/I_n=q^{\breve \ell(\pi(g))}$.

Moreover, it is proved in \cite[\S 4.4]{He-1} that for $g, g' \in G(F)$ with $\ell(\pi(g g'))=\ell(\pi(g))+\ell(\pi(g'))$, the multiplication map in $G(F)$ induces a bijection $$I_n g I_n \times_{I_n} I_n g' I_n \cong I_n g g' I_n.$$ Here $I_n g I_n \times_{I_n} I_n g' I_n$ be the quotient of $I_n g I_n \times I_n g' I_n$ by the action of $I_n$ defined by $a \cdot (z, z')=(z a \i, a z')$. In this case, $$\mathbbm{1}_{I_n g I_n} * \mathbbm{1}_{I_n g' I_n}=\mathbbm{1}_{I_n g g' I_n}.$$

\begin{corollary}\label{gen}
Let $n \ge 1$. The algebra $\CH_n$ is generated by $\mathbbm{1}_{I_n g I_n}$ for $g \in I$ and $\mathbbm{1}_{I_n \dot w I_n}$ for $w \in W$ with $\ell(w) \le 1$. In particular, $\CH_n$ is finitely generated.
\end{corollary}

\begin{proof}
Note that $\CH_n$ is spanned by $\mathbbm{1}_{I_n g I_n}$ for $g \in G(F)$.

Suppose that $g \in I \dot w I$ and $w=s_{i_1} \cdots s_{i_n} \tau$ for $s_{i_1}, \ldots, s_{i_n} \in \BS$ and $\tau \in W$ with $\ell(\tau)=0$. Then $g=g_1 \cdots g_n g'$, with $g_j \in I \dot s_{i_j} I$ for $1 \le j \le n$ and $g' \in I \dot \tau I$. Then $\mathbbm{1}_{I_n g I_n}=\mathbbm{1}_{I_n g_1 I_n} *\cdots *\mathbbm{1}_{I_n g_n I_n}* \mathbbm{1}_{I_n g' I_n}$.

For any $s \in \BS$ and $g \in I \dot s I$, we have $g=i_1 \dot s i_2$ for some $i_1, i_2 \in I$. Then $\mathbbm{1}_{I_n g I_n}=\mathbbm{1}_{I_n i_1 I_n} *\mathbbm{1}_{I_n \dot s I_n} *\mathbbm{1}_{I_n i_2 I_n}$.

For any $\tau \in W$ with $\ell(\tau)=0$ and $g \in I \dot \tau I$, we have $g=i \tau$ for some $i \in I$. Then $\mathbbm{1}_{I_n g_1 I_n}=\mathbbm{1}_{I_n i I_n} *\mathbbm{1}_{I_n \dot \tau I_n}$.

Note that $\Omega_\al$ is finitely generated. Let $\{\t_1, \cdots, \t_l\}$ be a generating set of $\Omega_\al$. Then $\CH_n$ is generated by $\mathbbm{1}_{I_n g I_n}$ for $g \in I$, $\mathbbm{1}_{I_n \dot s I_n}$ for $s \in \BS$ and $\mathbbm{1}_{I_n \dot \t_i I_n}$ for $1 \le i \le l$. Thus $\CH_n$ is finitely generated. 
\end{proof}

\subsection{The subgroup $\bP_{s,n}$}Let $\bP_{s,n}$ be as in \S \ref{sec:Ps}.

\begin{lemma}\label{normal}
For $n \ge 1$, $\breve P_{s, n}$ is a normal subgroup of $\breve I$.
\end{lemma}

\begin{proof}
Note that $\breve P_{s, n}$ is a normal subgroup of the parahoric subgroup $\breve P_{s}$. Since $\breve I \subset \breve P_{s}$, we have that $\breve P_{s, n}$ is stable under the conjugation action of $\breve I$.

It remains to show that $\breve P_{s, n} \subset \breve I$.

Recall that $\bv_s$ is the barycenter of the facet $\bal_s$ in the closure of the base alcove $\bal$. Let $\breve \Phi^+(G,S)$ be the set of positive roots in $\breve \Phi(G,S)$. Then, using \cite[\S 4.2.22]{BT2}, it follows that $0 \le \<\ba, \bv_s\> \le 1$ for any $a \in \breve\Phi^+(G,S)$.

By definition, $\breve P_{s, n}$ is generated by $T_n$ and $U_\ba(\bF)_{\bv_s, n}$. Let $u \in U_\ba(\bF)_{\bv_s, n}$. If $\ba \in \Phi^+(G,S)$, then the condition $\langle \ba, \bv_s\rangle + \phi_\ba(u) \geq n$ implies that $\phi_\ba(u) \ge n-1 \ge 0$. If $-a \in \Phi^+(G,S)$, then the condition $\<\ba, \bv_s\> +\phi_\ba(u) \ge n$ implies that $ \phi_\ba(u) \ge n \ge 1$. In both cases, $ U_\ba(\bF)_{\bv_s, n}\subset \bI$. Therefore $\breve P_{s, n} \subset \bI$.
\end{proof}

\begin{lemma}\label{gg-s}
Let $g \in G$ with $\pi(g)=s \in \BS$. Set $P_{s, n}=\breve P_{s, n}^\s$. Then for $n \ge 1$, $g I_n g \i I_n=I_n g I_n g \i=P_{s, n}$ and it is a normal subgroup of $I$.
\end{lemma}

\begin{proof}
For $n \ge 1$, $\breve I_n$ is stable under the conjugation action of $\dot s \breve I_n \dot s \i$ since $\dot s \breve I_n \dot s \i \subset \bI$. Therefore $\breve I_n (\dot s \breve I_n \dot s \i)=(\dot s \breve I_n \dot s \i) \breve I_n$. 

On the other hand, $T_n \subset \breve I_n$. 
Let $\ba \in \breve\Phi(G,S)$ and let $u \in U_\ba(\bF)_{\bv_s,n}$. So $\<\ba, \bv_s\>+ \phi_\ba(u) \geq n$.  We claim that $ u \in \bI_n \cup \dot \bs\bI_n \dot\bs\i$. Suppose $ u \notin \bI_n$. Then $\ba + \phi_\ba(u)-n$ is a negative affine root, so $\<\ba, \bv_s\>+ \phi_\ba(u) -n \leq 0$. Hence $\<\ba, \bv_s\>+ \phi_\ba(u)- n=0$. Then $\ba+\phi_\ba(u) - n$ belongs to $\breve\Phi_{\bv_s}$, the set of affine roots that vanish at $\bv_s$. Write $\breve\Phi_{\bv_s} = (\Phi_{\bv_s} \cap \Phi^+_{\af}(G,S)) \sqcup (\Phi_{\bv_s} \cap \Phi_{\af}^-(G,S))$. Then, since $\bW_\CX$ is the Weyl group of $\Phi_{\bv_s}$ and $s$ is the longest element of $\bW_{\CX}$, it follows that $s(\ba+\phi_\ba(u)-n)=s(\ba)+\phi_\ba(u)-n$ is a positive affine root. Then $u \in \dot s \breve I_n \dot s \i$. Therefore $$\breve P_{s, n}=\breve I_n (\dot s \breve I_n \dot s \i)=(\dot s \breve I_n \dot s \i) \breve I_n.$$

We have $g \breve I_n g \i=(i \dot s) \breve I_n (i s) \i=i (\dot s \breve I_n \dot s \i) i \i$ for some $i \in \breve \CI$. Thus $$(g \breve I_n g \i) \breve I_n =i (\dot s \breve I_n \dot s \i) i \i \breve I_n =i (\dot s \breve I_n \dot s \i)\breve I_n  i \i=i \breve P_{s, n} i \i=\breve P_{s, n}.$$

Thus $$(\breve P_{s, n}/\breve I_n)^\s=(g \breve I_n g \i \breve I_n/\breve I_n)^\s=g (\breve I_n g \i \breve I_n/\breve I_n)^\s=g (I_n g \i I_n/I_n).$$ Here the last equality follows from \S\ref{he1}.

Thus $g I_n g \i I_n=\breve P_{s, n}^\s=P_{s, n}$. This is a normal subgroup of $I$ since $\breve P_{s, n}$ is a normal subgroup of $\breve I$. As $P_{s, n} \subset I$ and $I_n$ is a normal subgroup of $I$, we also have $(g I_n g \i) I_n=I_n (g I_n g \i)$.
\end{proof}

\begin{proposition}
Let $n \ge 1$. Let $g, g' \in G$ with $\pi(g)=\pi(g')=s \in \BS$. The multiplication map on $G$ induces a surjective map $$I_n g I_n \times_{I_n} I_n g' I_n \to P_{s, n} g g' I_n.$$ Moreover, each fiber contains exactly $q^{\breve \ell(s)}$ elements.
\end{proposition}

\begin{proof}
By Lemma \ref{gg-s}, $I_n g I_n g' I_n=(I_n g I_n g \i) g g' I_n=P_{s, n} g g' I_n$.

If $\pi(g g')=s$, then $P_{s, n}=I_n (g g') I_n (g g') \i$ and $$P_{s, n} g g' I_n=I_n (g g') I_n (g g') \i (g g') I_n=I_n g g' I_n.$$ By \S\ref{he1}, $\sharp I_n g g' I_n/I_n=q^{\breve \ell(s)}$. Since the map $I_n g I_n \times_{I_n} I_n g' I_n/I_n \to P_{s, n} g g' I_n/I_n$ is equivariant under the left action of $I_n$ and $I_n$ acts transitively on $P_{s, n} g g' I_n/I_n$, all the fibers have the same cardinality and the cardinality equals to $$\sharp(I_n g I_n \times_{I_n} I_n g' I_n/I_n)/\sharp(P_{s, n} g g' I_n/I_n)=q^{2\breve \ell(s)}/q^{\breve \ell(s)}=q^{\breve \ell(s)}.$$

If $\pi(g g')=1$, then $g'=g \i i$ for some $i \in I$ and we have the following commutative diagram
\[
\xymatrix{
I_n g I_n \times_{I_n} I_n g' I_n \ar[r] \ar[d] & P_{s, n} g g' I_n \ar[d] \\
I_n g I_n \times_{I_n} I_n g \i I_n i \ar[r] & P_{s, n} i.
}
\]

Thus it suffices to consider the case where $g'=g \i$. Since $n \ge 1$, by Lemma \ref{normal} and Lemma \ref{gg-s}, $g I_n g \i \subset P_{s, n} \subset I$. Since $I_n$ is a normal subgroup of $I$, $I_n$ is stable under the conjugation action of $g I_n g \i$. Thus $g \i I_n g$ is stable under the conjugation action of $I_n$. Conjugation by $g \i$, we have that $g \i I_n g \cap I_n$ is a normal subgroup of $I_n$. Thus for any $p \in P_{s, n}$, the inverse image of $I_n p I_n$ in $I_n g I_n \times_{I_n} I_n g' I_n$ equals to $$\{(I_n g a, b g \i I_n); a, b \in I_n/(g \i I_n g \cap I_n), g a b g \i \in I_n p I_n\}/I_n.$$

Let $p' \in P_{s, n}$. Then since $P_{s, n}=(g I_n g \i) I_n$, we have $p' I_n=(g i g \i) p I_n$ for some $i \in I_n$. Note that $I_n$ is stable under the conjugation action of $g i g \i \in P_{s, n} \subset I$. Hence $(g i g \i) I_n p I_n=I_n (g i g \i) p I_n=I_n p' I_n$ and the inverse image of $I_n p' I_n$ in $I_n g I_n \times_{I_n} I_n g' I_n$ equals to $$\{(I_n g i a, b g \i I_n); a, b \in I_n/(g \i I_n g \cap I_n), g a b g \i \in I_n p I_n\}/I_n.$$

In particular all the fibers have the same cardinality. So the cardinality of each fiber equals to $$\sharp(I_n g I_n \times_{I_n} I_n g' I_n/I_n)/\sharp(P_{s, n}/I_n)=q^{2\breve \ell(s)}/q^{\breve \ell(s)}=q^{\breve \ell(s)}.$$

The statement is proved.
\end{proof}

\subsection{Proof of Theorem \ref{mult}}
We choose a representative $g$ for each $I_n \times I_n$-orbit on $G(F)$. We denote the set of representatives by $\CY$. Then the set $\{\mathbbm{1}_{I_n g I_n}; g \in \CY\}$ is a basis of $\CH_n$ as a free $\BZ$-module. In particular, the set $\{\mathbbm{1}_{I_n g I_n}; g \in \CY\}$ generates $\CH_n$ as an algebra. 

For any $g_1, g_2 \in \CY$, we have \[\mathbbm{1}_{I_n g_1 I_n} * \mathbbm{1}_{I_n g_1 I_n}=\sum_{g_3 \in \CY} c_{g_1, g_2, g_3} \mathbbm{1}_{I_n g_3 I_n} \in \CH_n\] for some $c_{g_1, g_2, g_3} \in \BZ$. We denote this relation by $(*_{g_1, g_2})$. It is tautological that the equalities $(*_{g_1, g_2})$ for $g_1, g_2 \in \CY$ form a set of relations for the algebra $\CH_n$. 

By Corollary \ref{gen}, the algebra $\CH_n$ is generated by $\mathbbm{1}_{I_n g I_n}$ for $g \in \CI$ with $\ell(\pi(g)) \le 1$. Thus the equalities $(*_{g_1, g_2})$ for $g_1, g_2 \in \CY$ with $\ell(\pi(g_1)) \le 1$ form a set of relations for the algebra $\CH_n$. 

If $\ell(\pi(g_1))=0$, then for any $g_2 \in \CY$, $\ell(\pi(g_1 g_2))=\ell(\pi(g_2))$. Thus the equality $(*_{g_1, g_2})$ is obtained from the relations $(0)$ and $(1)$ in Theorem \ref{mult}. 

If $\ell(\pi(g_1))=1$, then $\pi(g_1)=s$ for some $s \in \BS$. Let $w=\pi(g_2)$. If $s w>w$, then the equality $(*_{g_1, g_2})$ is obtained from the relations $(0)$ and $(1)$ in Theorem \ref{mult}. If $s w<w$, then $\pi(g_1 g_2)=\pi(g_2)$ or $\pi(g_1 g_2)<\pi(g_2)$. In either case, the equality $(*_{g_1, g_2})$ is obtained from the relations $(0)$ and $(2)$ in Theorem \ref{mult}. 

Theorem \ref{mult} is proved. 

\subsection{The Howe-Tits presentation of $\CH_n$}

In \cite{How85}, Howe discovered a nice presentation for the Hecke algebra $\CH_n$. This presentation was later generalized to split groups by the first-named author in \cite{Gan15}. This nice presentation of $\CH_n$ has found applications in the representation theory of $p$-adic groups. For instance, this presentation was used to establish a variant of a Hecke algebra isomorphism of Kazhdan for sufficiently close local fields, which in turn was used to study the local Langlands correspondence for connected reductive groups in characteristic $p$ with an understanding of the local Langlands correspondence of such groups in characteristic 0 (see \cite{Bad02, Lem01,  Gan15, ABPS14, GV17}). 

Before stating the theorem, we first introduce some structure constants. 

For $\tau, \tau' \in \Omega_\al$ and $s \in \BS$, let 
\begin{align}\label{cs}
    c_{\tau, \tau'} &= m(\tau)m(\tau') m(\tau\tau')\i\\\nonumber
    c_{\tau, s} &= m(\tau)m(s)m(\tau)\i m(\tau s\tau\i)\i.\nonumber
\end{align}

Recall that the Tits' axiom (T3) (see \cite[\S 1.2.6]{BT1}) says that $m(s)Im(s)^{-1} \subset I \cup I m(s) I$. In particular, if $g \in I$ but $g \notin I\cap m(s)Im(s)^{-1}$, this axiom implies that $m(s)gm(s)\i \in Im(s)I$. Hence there exist $g_1, g_2$ in $I$ such that $m(s)g m(s)^{-1} =g_1 m(s) g_2$.  
We have the following theorem.

\begin{theorem}\label{thm:howe-tits} 
Let $\CT$ be a Tits group of $W$ and $\{m(w)\}_{w \in W}$ is a Tits cross-section of $W$ in $\CT$. The Hecke algebra $\CH_n$ has generators
\begin{enumerate}
\item $\mathbbm{1}_{I_n m(s) I_n}, s \in \BS$,
\item $\mathbbm{1}_{I_n m(\tau) I_n},\; \tau \in \Omega_\al$,
\item $\mathbbm{1}_{I_n g I_n}, g \in I$,
\end{enumerate}
subject to the following relations:
\begin{enumerate}[(A)]
\item \begin{enumerate}[(i)]
\item For $s, s'$ distinct elements of $\BS$ with $s \cdot s'$ of order $k(s,s')$, 
\[ \underbrace{\mathbbm{1}_{I_n m(s) I_n}*\mathbbm{1}_{I_n m(s') I_n} *\cdots}_{k(s, s') \text{ factors}} = \underbrace{\mathbbm{1}_{I_n m(s') I_n}*\mathbbm{1}_{I_n m(s) I_n} *\cdots.}_{k(s, s') \text{factors}} \]
\item For $s \in \BS$, $\mathbbm{1}_{I_n m(s) I_n} * \mathbbm{1}_{I_n m(s) I_n} * \mathbbm{1}_{I_n m(s)^2 I_n} = q^{\breve{l}(s)} \displaystyle{\sum_{x \in P_{s,n}/I_n} \mathbbm{1}_{I_n x I_n}}.$
\end{enumerate}
\smallskip
\item
\begin{enumerate}[(i)]
\item For $\tau, \tau' \in \Omega_\al$, $\mathbbm{1}_{I_n m(\tau) I_n} * \mathbbm{1}_{I_n m(\tau') I_n} = \mathbbm{1}_{I_n c_{\tau, \tau'} I_n}* \mathbbm{1}_{I_n m(\tau\tau') I_n}$. 
\item For $\tau \in \Omega_\al$ and $s \in \BS$, $$\mathbbm{1}_{I_n m(\tau) I_n} * \mathbbm{1}_{I_n m(s) I_n} *\mathbbm{1}_{I_n m(\tau\i) I_n} *  \mathbbm{1}_{I_n c_{\t, \tau\i} I_n} = \mathbbm{1}_{I_n c_{\tau,s }I_n}* \mathbbm{1}_{I_n m(\tau s \tau^{-1}) I_n}.$$
\item For $\tau \in \Omega_\al$ and $g \in I$, $$\mathbbm{1}_{I_n m(\tau) I_n} * \mathbbm{1}_{I_n g I_n} * \mathbbm{1}_{I_n m(\tau\i) I_n} *  \mathbbm{1}_{I_n (c_{\t, \tau\i})\i I_n} = \mathbbm{1}_{I_n m(\tau)g m(\tau)\i I_n}.$$
\end{enumerate}
\smallskip
\item \begin{enumerate}[(i)]
\item $\mathbbm{1}_{I_n}$ is the identity element of $\CH(G(F), I_n)$.
\item For $g,g' \in I$, $\mathbbm{1}_{I_n g I_n}* \mathbbm{1}_{I_n g' I_n} = \mathbbm{1}_{I_n gg' I_n}.$
\item For $s \in \BS$ and for $g \in I \cap m(s)Im(s)^{-1}$, 
\[ \mathbbm{1}_{I_n m(s) I_n}* \mathbbm{1}_{I_n g I_n} = \mathbbm{1}_{I_n m(s)gm(s)^{-1} I_n} *  \mathbbm{1}_{I_n m(s) I_n}\]
\item For $s \in \BS$ and for $g \in I \backslash (I\cap m(s)Im(s)^{-1})$, let $g_1, g_2$ are elements of $I$ such that $m(s)g m(s)^{-1} =g_1 m(s) g_2$. Then 
\[ \mathbbm{1}_{I_n m(s) I_n}* \mathbbm{1}_{I_n g I_n} * \mathbbm{1}_{I_n m(s) I_n} * \mathbbm{1}_{I_n m(s)^2 I_n}= q^{\breve l(s)} (\mathbbm{1}_{I_n g_1 I_n} *  \mathbbm{1}_{I_n m(s) I_n}* \mathbbm{1}_{I_n g_2 I_n}).\]
\end{enumerate}
\end{enumerate}
\end{theorem}
\begin{proof}
Let $\hat \CH_n$ be the quotient of the free $\BZ$-algebra generated by the elements (1) - (3), by the subalgebra generated by relations (A) - (C) stated in the theorem. For clarity, to distinguish the elements of $\hat \CH_n$ from the elements of $\CH_n$, we denote the generators of $\hat \CH_n$ as $\HI_{I_n m(s) I_n}, s\in \BS$, $\HI_{I_n m(\tau) I_n}, \tau\in \Omega_\al$, $\HI_{I_n g I_n}, g \in I$.  By Theorem \ref{mult}, the relations (A)-(C) in Theorem \ref{thm:howe-tits} are satisfied for $\CH_n$. Thus we have an algebra homomorphism $\hat \CH_n \rightarrow \CH_n$. This map is surjective by Corollary \ref{gen}.  

For $w \in W$, write $w  = w_1 \tau$ for $w_1 \in W_{\af}$ and $\tau \in \Omega_\al$. 
Let $m(w_1) = m(s_{i_1}) \cdots m(s_{i_c})$ for a chosen reduced expression of $w_1$ and let $m(w) = m(w_1)m(\tau) \in \CT$. Define $\mathbbm{\hat 1}_{I_nm(w)I_n}= \HI_{I_n m(s_{i_1})I_n}*  \cdots \HI_{I_n m(s_{i_c})I_n} * \HI_{I_n m(\tau)I_n}$. Then relation (A)(i) implies that this expression is independent of the choice of reduced expression for $w_1$. For $g \in G(F)$, write $g = x m(w) y$ for $w \in W$ and $x, y \in I$. Define $\mathbbm{\hat 1}_{I_n g I_n}= \HI_{I_n x I_n} *\HI_{I_nm(w)I_n} *\HI_{I_n y I_n}$. 

We show that 

(a) Let $w \in W$ and $x, y, x_1, y_1 \in I$ such that $x m(w) y=x_1 m(w) y_1$. Then  \[\HI_{I_nxI_n} *\mathbbm{\hat 1}_{I_nm(w)I_n} *\HI_{I_n yI_n} =  \HI_{I_nx_1I_n} *\HI_{I_nm(w)I_n}*\HI_{I_ny_1I_n}.\]

Since $yy_1^{-1} = m(w)^{-1}(x^{-1}x_1)m(w)$,  $yy_1^{-1} \in \displaystyle{I \cap  m(w)^{-1}I m(w)}$. Therefore,
 \begin{align*}
\HI_{I_n m(w)I_n} * \HI_{I_nyy_1^{-1}I_n}  =  \HI_{I_n m(s_{i_1})I_n}* \ldots \HI_{I_n m(s_{i_c})I_n} * \HI_{I_n m(\tau) yy_1^{-1}m(\tau)\i I_n} * \HI_{I_n m(\tau)I_n}.
\end{align*}
using relations B(iii) and C(i). It is easy to check that
$a \in  \Ad (m(s_{i_1})\ldots  m(s_{i_c}))^{-1} (I) \cap I$  implies that  $a \in \Ad (m(s_{i_c}))( I) \cap I$ and  $\Ad (m(s_{i_c}))(a) \in I \cap \Ad (m(s_{i_1})\ldots  m(s_{i_{c-1}}))^{-1} (I).$  
Using the above and relation C(iii) repeatedly, we have 
\begin{align*}
\HI_{I_nm(w)I_n} * \HI_{I_nyy_1^{-1}I_n} & =\HI_{I_n m(s_{i_1})I_n}* \HI_{I_n m(s_{i_2})I_n}*\ldots.  \HI_{I_n m(s_{i_c})m(\tau) yy_1\i m(\tau)\i m(s_{i_c})\i I_n}*\\
& \quad *  \HI_{I_n m(s_{i_c})I_n}* \HI_{I_n m(\tau)I_n}\\
& = \HI_{I_n m(w) yy_1\i m(w)\i I_n}* \HI_{I_nm(w)I_n}\\
& = \HI_{I_nx^{-1}x_1I_n} * \HI_{I_nm(w)I_n}.
\end{align*}
Now (a) follows from C(i) and C(ii).

In particular, the element $\mathbbm{\hat 1}_{I_n g I_n}$ is well-defined. 

We prove that Relation (0) of Theorem \ref{mult} holds for $\hat \CH_n$. Let $g=x m(w) y$ with $x, y \in I$ and $g' \in I_n g I_n$. Then $g'=x' m(w) y'$ for some $x', y' \in I$ with $x' \in I_n x$ and $y' \in y I_n$. By (a) and relation (C)(i), (ii), we have 
\[\mathbbm{\hat 1}_{I_n g I_n}= \HI_{I_n x I_n} *\HI_{I_nm(w)I_n} *\HI_{I_n y I_n}=\HI_{I_n x' I_n} *\HI_{I_nm(w)I_n} *\HI_{I_n y' I_n}=\mathbbm{\hat 1}_{I_n g' I_n}.
\]

We prove that Relation (1) of Theorem \ref{mult} holds for $\hat \CH_n$. We need to show that if $l(\pi(gg')) = l(\pi(g)) + l(\pi(g'))$, then
\begin{align}\label{lade}
    \HI_{I_n g I_n} * \HI_{I_ng'I_n} = \HI_{I_ngg'I_n}.
\end{align}
  To prove this claim, we may easily reduce ourselves to the case when $l(\pi(g)) \leq 1$. If $l(\pi(g))=0$, then \eqref{lade} follows from relations B(i), B(ii), B(iii), C(i) and C(ii). If $l(\pi(g) = 1$, we may assume $g = m(s)$ for a suitable $s \in \BS$. Let $\pi(g') = w$. Write $g' = x' m(w) y'$ for $x', y' \in I$. Then $gg' = m(s)x'm(w)y'$. Since $l(sw) = l(s)+1$, we see that $m(s) I m(w) \subset I m(s)m(w) I$. In particular, $I \subset m(s)\i I m(s) m(w) I m(w)\i$. Write $x'= x_1'x_2'$ for $x_1' \in  m(s)\i I m(s)$ and $x_2' \in m(w) I m(w)\i$. Then \[gg' = m(s) x_1' m(s)\i m(s) m(w) m(w)\i x_2' m(w) y'\] and
  \begin{align*}
      \HI_{I_n m(s) g' I_n} &= \HI_{I_n m(s) x_1 m(s)\i I_n} *\HI_{I_n m(s) m(w) I_n} * \HI_{I_n m(w)\i x_2 m(w) y' I_n}.
  \end{align*}
In this case, \eqref{lade} holds using relations C(i), C(ii) and C(iii). 

We prove that relation (2) of Theorem \ref{mult} holds for $\hat \CH_n$. We may assume $g = m(s)$. Let $\pi(g') = w$. In this case, $l(sw)<l(w)$. Write $g'= x'm(w)y'$. Then $m(s)x'm(w) \in I m(w) I \sqcup I m(s) m(w) I$. Now $m(s)x'm(w) \in I m(s) m(w) I$ if and only if $\pi(gg') <\pi(g')$. Further, writing $x'= x_1'x_2'$ for $x_1' \in  m(s)\i I m(s)$ and $x_2' \in m(w) I m(w)\i$, using C(iii) and C(ii), we have
\begin{align}\label{R2A}
      \HI_{I_n m(s)  I_n} * \HI_{I_n x' I_n} * \HI_{I_n m(w) I_n } * \HI_{I_n y' I_n }&= \HI_{I_n m(s) x_1 m(s)\i I_n} *\HI_{I_n m(s) I_n}\\& * \HI_{I_n m(w) I_n} * \HI_{I_n m(w)\i x_2 m(w) y' I_n}. \nonumber
  \end{align}
  Since $l(sw)<l(w)$, we have $w = sw'$ for a suitable $w' \in W$. Then $\HI_{I_n m(w) I_n } = \HI_{I_n m(s) I_n }* \HI_{I_n m(w_1) I_n }$. Using this in \eqref{R2A} and using Relation A(ii) finishes the proof of Relation (2) when $\pi(gg') <\pi(g')$. Next,  $m(s)x'm(w) \in I m(w) I$ if and only if $\pi(gg') = \pi(g')$. In this case, we may write $x' = x_1' m(s)\i x_2'$ where $x_1' \in m(s)\i I m(s)$ and $x_2' \in m(w) I m(w)\i$. Then 
 \[m(s)g' = (m(s) x_1'm(s)\i) m(w) (m(w)\i x_2' m(w) y')\]
  Now
  \begin{align*}
       \HI_{I_n m(s)  I_n} * \HI_{I_n g' I_n } &= q^{l(s)} (\HI_{I_n m(s) x_1' m(s)\i I_n} *\HI_{I_n x_2' I_n} * \HI_{I_n m(w) I_n } * \HI_{I_n y' I_n })\\
       &=q^{l(s)}( \HI_{I_n m(s) x_1 m(s)\i I_n} * \HI_{I_n m(w) I_n } * \HI_{I_n m(w)\i x_2' m(w) I_n} *\HI_{I_n y' I_n})\\
       & = q^{l(s)} (\HI_{I_n m(s) g' I_n}).
  \end{align*}
  In the above, the first equality uses C(iv) and C(i), the second one uses C(iii) and C(i).
  
  We have verified that relations (0) - (2) of Theorem \ref{mult} hold for $\hat \CH_n$. This concludes the proof of the theorem.
  \end{proof}

\section{Tits groups over $\breve F$}\label{sec:TGbF}

\subsection{The Tits group of the relative Weyl group over $\bF$}

We begin with a discussion on the Tits group of the relative Weyl group of $G$ over $\bF$, which is probably well-known, but a reference discussing this does not seem to be available in literature. 

Let $G$ be a connected, reductive group over $F$. In this section, we prove the existence of the Tits group of a finite relative Weyl group of $G_\bF$. 

We consider a Steinberg pinning $(x_\ta)_{\ta \in \tilde\Delta_0}$ of $G_\tF$ relative to $S$ (see \cite[\S 4.1.3]{BT2}). It has the following properties: 
\begin{enumerate}[(1)]
\item $x_\ta: \BG_a \rightarrow U_\ta$ is a $\tF$-isomorphism.
\item For each $\ta \in \tilde\Delta_0$ and $\gamma \in \Gal(\tF/\bF)$, $x_{\gamma(\ta)} = \gamma \circ x_\ta\circ \gamma^{-1}$.
\end{enumerate}
This extends to a Chevalley-Steinberg system of pinnings $x_\ta: \BG_a \xrightarrow{\cong} U_\ta$ for all $\ta \in \tilde\Phi(G,T)$, which is compatible with the action of $\Gal(\tF/\bF)$.  From this, we get a set of pinnings for $\ba \in \breve\Phi(G,S)$, which we briefly recall. Let $U_\ba$ be the root subgroup of the root $a$. When $2a$ is not a root, we have $x_\ba: \Res_{\bF_{\ba}/\bF}\BG_a \xrightarrow{\cong}   U_\ba$. 
If $2\ba$ is a root, let $H_0(\bF_\ba, \bF_{2\ba}) = \{(u,v) \in \bF_\ba \times \bF_\ba\;|\; u \cdot \gamma(u) = v+\gamma_0(v)\}$, where $\gamma_0$ is the non-trivial $\bF_{2\ba}$-automorphism of $\bF_\ba$. We have $$x_\ba:  Res_{\bF_{2\ba}/\bF} H_0(\bF_\ba, \bF_{2\ba})\xrightarrow \cong U_\ba.$$

For any $\ta \in \tilde\Delta_0$, let 
\begin{align}\label{tilderep}
n_{s_\ta}:=x_\ta(1) x_{-\ta}(1) x_{\ta}(1).
\end{align}
We note here that we have used the convention of \cite[\S 4.1.5]{BT2}, and $n_{s_\ta} \in N_G(S)(\bF)$ (This is different from the convention used in \cite{Spr09} where $n_{s_\ta}:=x_\ta(1) x_{-\ta}(-1) x_{\ta}(1)$).

Now, let $\ba \in \breve \D_0$. If $2\ba$ is not a root, we set
\begin{align}\label{Rep1}
n_{s_\ba} := x_\ba(1)x_{-\ba}(1)x_\ba(1).
\end{align}

Next, suppose $2\ba$ is a root. By \cite[ Chapter V, \S 4,  Proposition 7]{Ser79}, there exists $c \in \bF_\ba$ such that $c \gamma_0(c)=2$. Set \begin{equation}\label{Rep2'}
    n_{s_\ba} = x_\ba(c,1) x_{-\ba}(c,1)x_\ba(c,1).
\end{equation}

Note that if the residue characteristic of $F$ is not 2, such a $c$ in fact lies in $\fkO_\bF^\times$, and  when the characteristic of $F$ is 2, $c =0$. By \cite[\S 4.1.11]{BT2}, we have 
\begin{align}\label{Rep2}
n_{s_\ba}= \prod n_{s_\ta}^{-1} n_{s_{\ta'}}n_{s_{\ta}}^{-1}.
\end{align}
where the product is indexed by the family of sets $\{\ta, \ta'\}$ with $\ta, \ta' \in \tilde\Phi(G,T)$ such that $\ta+\ta'$ is a root and $\ta|_S =\ta'|_S = \ba$.

Let $\breve \CT_{\fin}$ be the group generated by the elements $\{n_{s_\ba}\;|\; \ba \in \breve\Delta_0\}$. Note that $S_2 = \langle \ba^\vee(-1)\;|\; \ba \in \breve \Delta_0\rangle$ is contained in $\CT_\fin$. Then the elements $\{n_{s_{\ba}}\;|\; \ba \in \breve\Delta_0\}$ satisfy Coxeter relations is a consequence of \cite[\S 6.1.3]{BT1} (see \cite[Proposition IV.6]{AHHV}). Furthermore, we have that 
\[n_{s_\ba}^2 =\begin{cases} \ba^\vee(-1), & \text{ if $2\ba$ is not a root}; \\ 1, & \text{ if $2\ba$ is a root}.\end{cases}\] 
Let $\breve S_2 = \langle \ba^\vee(-1)\; |\; \ba \in \breve\Phi(G,S) \rangle$. Then $\CT_{\fin}$ fits into a short exact sequence
\[1 \rightarrow  \breve S_2 \rightarrow \breve\CT_{\fin} \rightarrow W(G,S) \rightarrow 1. \]

\subsection{An example of $G_\bF$ for which $\breve\CT$ does not exist}\label{sec:example} In this subsection, we will give an example of a  wildly ramified unitary group over $\bF$ for which the Tits group $\breve\CT$ does not exist.

Let $F = \BQ_2, \bF = \BQ_2^{\text{un}}, \tF = \bF(\sqrt-1)$. Then $\tF$ is a wildly ramified quadratic extension of $\bF$.  Let $G$ be a connected reductive group over $\bF$ with $G_\bF = \U_6\subset \Res_{\tF/\bF}\GL_6$. Let $\gamma$ denote the generator of $\Gal(\tF/\bF)$.  Then 
\begin{gather*} \breve\Phi(G,S) = \{\pm e_i \pm e_j\;|\; 1\leq i<j\leq 3\} \cup \{\pm2 e_i\;|\; 1 \leq i \leq 3\}, \\ 
\breve\Phi_{\af}(G,S) = \{ \pm e_i \pm e_j +\frac{1}{2}\BZ\;|\;1 \leq i<j \leq 3\} \cup \{\pm 2e_i+\BZ\;|\; 1 \leq i\leq 3\}. 
\end{gather*}
The hyperplanes with respect to the roots $\ba_1:= e_1-e_2, \ba_2:=e_2-e_3, \ba_3:= 2e_3, \ba_0:=-e_1 - e_2 +\frac{1}{2}$ form a alcove in $\CA(S,\bF)$ which we denote as $\bal$. Let $\breve\Delta_0 =\{\ba_i\;|\; 1\leq i \leq 3\}$ and $\breve\Delta = \breve\Delta_0 \cup \{ \ba_0\}$. Let $\bs_i = s_{\ba_i}, 0 \leq i \leq 3$. 

Suppose $\breve\CT$ can be defined. Then $\breve\CT$ contains representatives $m(\bs_i), 0 \leq i\leq 3$ that satisfy Coxeter relations and such that $m(\bs_i)^2 = (\bb_i)_*^\vee(-1) \in \breve S_2$.  

The Coxeter relations involving the reflection $\bs_0$ are $\bs_0\bs_1= \bs_1\bs_0$, $\bs_0\bs_2\bs_0= \bs_2\bs_0\bs_2$ and $\bs_0\bs_3=\bs_3\bs_0$. 
Additionally, we have $\bs_0^2=1$.

Let $t \in T(\bF)$. We may write $t = \diag(d_1, d_2, d_3, \gamma(d_3)\i, \gamma(d_2)\i, \gamma(d_1)\i)$. Then 
\begin{align*}
    \bs_1(t) &= \diag(d_2, d_1, d_3, \gamma(d_3)\i, \gamma(d_1)\i, \gamma(d_2)\i),\\
    \bs_2(t) &=  \diag(d_1, d_3, d_2, \gamma(d_2)\i, \gamma(d_3)\i, \gamma(d_1)\i),\\
    \bs_3(t) &= \diag(d_1, d_2, \gamma(d_3)\i, d_3, \gamma(d_2)\i, \gamma(d_1)\i),\\
    \bs_0(t) &= \diag(\gamma(d_2)\i, \gamma(d_1)\i, d_3, \gamma(d_3)\i, d_1, d_2).
\end{align*}
To see the last equality, we note that a reduced expression for the image of $\bs_0$ in $W(G,S)$ is $\bs_2\bs_3\bs_2\bs_1\bs_2\bs_3\bs_2$. This equality can also be seen by noting that the image of $\bs_0$ in $W(G,S) \subset W(G,T)$ represents the permutation $(1,5)(2,6)$ in the symmetric group $S_6$. 

For $\sigma$ a permutation in the symmetric group $S_6$, let $g_\sigma$ denote the corresponding permutation matrix in $\GL_6(\tF)$ whose entries are all 0 or 1. The element $m(\bs_0) \in \U_6(\bF) \subset \GL_6(\tF)$ can be written as a product $t_0 \cdot g_{(1,5)(2,6)}$  where $$t_0 = \diag(d_1, d_2, d_3, \gamma(d_3)\i, \gamma(d_2)\i, \gamma(d_1)\i).$$

Now $m(\bs_0)^2 \in \breve S_2$ implies that $d_2 = \pm \gamma(d_1)$. Next, the relation $m(\bs_0)m(\bs_1) = m(\bs_1)m(\bs_0)$ implies that $t_0 g_{(1,5)(2,6)} = \bs_1(t_0) m(\bs_1) g_{(1,5)(2,6)} m(\bs_1)\i$. So \[\bs_1(t_0) t_0\i =  g_{(1,5)(2,6)} m(\bs_1)  g_{(1,5)(2,6)}m(\bs_1)\i.\] Since the right side of the expression above maps to the trivial element of $\bW$, it is an element of $\breve S_2$; in particular, it is a diagonal matrix with entries $\pm 1$. Since \[\bs_1(t_0)t_0^{-1} = \diag(\pm \gamma(d_1) d_1\i, \pm d_1 \gamma(d_1)\i, 1,1,\pm \gamma(d_1) d_1\i,\pm d_1 \gamma(d_1)\i)\] we have that $\gamma(d_1) = \pm d_1$. Similarly, $m(\bs_0) m(\bs_3) = m(\bs_3) m(\bs_0)$ implies that $d_3 =\pm 1$. The action of $m(\bs_0)$  on $V$ is given by $s_{e_1+e_2}(v) - val(d_1)(e_1+e_2)^\vee(v)$, where $val$ is normalized so that $val(\bF)=\BZ$. Hence $val(d_1) = \frac{1}{2}$ (since $\bs_0$ is the reflection with respect to the vanishing hyperplane of the affine root $-e_1-e_2 + \frac{1}{2}$). Hence $\gamma(d_1) \neq d_1$. So $\gamma(d_1) = -d_1$. 

In conclusion, the assumption that $\breve\CT$ can be defined implies that there exists $d_1 \in \tF$ with $val(d_1)=\frac{1}{2}$ and $Tr_{\tF/\bF}(d_1) = 0$. 
However, there does not exist an element of $\tF$ with valuation $\frac{1}{2}$ and with trace 0. To see this, note that any element of $\tF$ is of the form $a+b \sqrt -1$ for $a, b \in \bF$. Then $\gamma(a+b\sqrt -1) = a - b \sqrt -1$. If $\gamma(a+b\sqrt -1) = -a-b\sqrt -1$ then $a=0$. Further $\sqrt -1$ is a unit in $\tF$, so $val( b\sqrt -1) = val(b) \in \BZ$. 

This gives a contradiction and proves that $\breve\CT$ cannot be defined. 

\smallskip

In the rest of this section, we will show that the Tits group of $\bW$ exists if $G$ splits over $\breve F$. 

\subsection{Affine pinning}\label{Affpinning} Recall that we have chosen a special vertex $\bv_0$ and we view $\bal \subset V$. For each reflection $\bs$ in $\bBS$ there is a unique affine root such that the reflection in the hyperplane with respect to this affine root is $\bs$. Let $\breve\Delta$ denote this collection of affine roots. The Weyl chamber in $V$ that contains $\bal$ determines a set of simple roots of $\breve\Phi(G,T)$ which we denote as $\breve\Delta_0$. Note that $\breve\Delta_0 \subset \breve\Delta$. 

We consider a pinning $\{x_\ba\}_{\ba \in \tilde\Delta_0}$ of $G_\bF$ relative to $T$ (see \cite[\S 4.1.3]{BT2}), that is, for each $\ba \in \breve\Delta_0$, we fix a $\bF$-isomorphism $x_\ba: \BG_a \rightarrow U_\ba$, where $U_\ba$ is the root subgroup of the root $\ba$. 
This extends to a Chevalley system of pinnings $x_\ba: \BG_a \xrightarrow{\cong} U_\ba$ for all $\ba \in \breve\Phi(G,T)$.

Let $\ba \in \breve\Delta \backslash \breve\Delta_0$ and let $\bb\in \breve\Phi(G,T)$ be the gradient of $\ba$. 
Let $m_{\varpi_F}: \BG_a \rightarrow  \BG_a$ denote $\bF$-morphism given by multiplication by $\varpi_{F}$, where $\varpi_F$ is a uniformizer of $F$. Define $x_{\ba}:= x_{\bb} \circ m_{\varpi_{F}}$. Note that $x_{\ba}$ is a $\bF$-isomorphism from $\BG_a$ to $U_{\bb}$.

The set $\{x_\ba\;|\; \ba \in \breve\Delta\}$ is called an \textit{affine pinning} of $G_{\bF}$. For $\ba \in \breve \D$ define 
\begin{align}\label{Rep}
n_{s_\ba} := x_\ba(1)x_{-\ba}(1)x_\ba(1).
\end{align}
We note here again that we use the convention of \cite[\S 4.1.5]{BT2} and $n_{s_\ba} \in N_G(T)(\bF)$. 

\begin{lemma}\label{TitsWaf}
The set $\{n_\bs\;|\; \bs \in \bBS\}$ satisfies the Coxeter relations.  
\end{lemma}

\begin{remark}
For a different proof of the Coxeter relations for the affine Weyl group of a split reductive group, see \cite[Proposition 3.1]{Gan15}. 
\end{remark}

\begin{proof} Let $\bs = s_\ba$ and $\bs' = s_{\ba'}$ for $\ba, \ba' \in \breve\Delta$ with gradients $\bb,\bb'$ respectively. Let $\breve\Phi_{\bb,\bb'} \subset \breve\Phi(G,T)$ denote the rank 2 root system spanned by $\bb,\bb'$. If $\breve\Phi_{\bb,\bb'}$ is a product of rank 1 root systems, then the Coxeter relation is obvious, so we may and do assume that $\breve\Phi_{\bb,\bb'}$ is irreducible.

Set $k = k(\bs,\bs')$. We put the elements of $\breve\Phi_{\bb,\bb'}$ in ``circular order" as required in \cite[Proposition 6.1.8]{BT1}, that is, we can enumerate the elements of $\breve\Phi_{\bb,\bb'}$ as $\bb_1, \bb_2, \cdots, \bb_{2k}$ so that $\bb_1=\bb, \bb_k=\bb'$, and for $1<i<2k$,
\[\breve\Phi_{\bb,\bb'} \cap (\BQ_+\bb_{i-1} + \BQ_+ \bb_{i+1}) = \{\bb_{i-1}, \bb_i, \bb_{i+1}\}.\]

By \cite[\S 6.1.3]{BT1}, for any $u \in U_{-\bb}(\bF)\backslash\{1\}$ and $u' \in U_{-\bb'}(\bF)\backslash\{1\}$, there exists unique triples $(u_1, m(u), u_2) \in U_{\bb}(\bF) \times N_G(S)(\bF) \times U_{\bb}(\bF)$ and $(u'_1, m(u'), u'_2) \in U_{\bb'}(\bF) \times N_G(S)(\bF) \times U_{\bb'}(\bF)$ such that $u = u_1m(u)u_2$ and $u'=u'_1 m(u') u'_2$. By \cite[Proposition 6.1.8, Part (9)]{BT1},
\begin{align}\label{genericcoxeter}
   m(u) \cdot m(u') \i \cdots = m(u') \i \cdot m(u) \cdots,
\end{align}
where each side has $k$ factors.

It is clear from equation \eqref{Rep} that there exist $u \in U_{-\bb_*}$ and $u' \in U_{-\bb'_*}$ such that $n_{s_\ba} = m(u)$ and $n_{s_{\ba'}} = m(u')^{-1}$. Now the statement follows from \eqref{genericcoxeter}.
\end{proof}

Now we prove the following existence result of Tits group over $\breve F$. 

\begin{proposition}\label{prop:tits}
Suppose that $G$ is split over $\breve F$. Let $\{x_{\ba}\;|\; \ba \in \breve \D\}$ be an affine pinning of $G_{\breve F}$ and $\{n_{s_\ba}\;|\; \ba \in \breve\Delta\}$ be as in \eqref{Rep}. Let $\breve \CT$ be the group generated by $\breve S_2$, $\{n_{s_\ba}\;|\; \ba \in \breve\Delta\}$ and $\blambda(\varpi_F)$ for $\blambda \in X_*(T)$. Then $\breve \CT$ is a Tits group of $\bW$. 
\end{proposition}

\begin{proof}
By direct calculation, $n_{\bs_\ba}^2=\bb^\vee(-1)$, where $\bb$ is the gradient of $\ba$. 

Now we define the lifting $n_{\bw}$. 

If $\bw \in \bW_{\af}$, then we set $n_{\bw}=n_{\bs_{i_1}} \cdots n_{\bs_{i_k}}$, where $\bs_{i_1} \cdots \bs_{i_k}$ is a reduced expression of $\bw$ in $\bW$. By Lemma \ref{TitsWaf}, the element $n_{\bw}$ is independent of the choice of reduced expression. If $\breve \t \in \Omega_{\bal}$, we may write $\breve \t$ as $t_{\breve \l} \by$ with $\breve \l \in X_*(T)$ and $\by \in \bW_0$. We then set $n_{\breve \t}=\breve \l(\varpi_F) n_{\by}$. Note that any element $\bw \in \bW$ is of the form $\bw=\bw_1 \breve \t$ for some $\bw_1 \in \bW_{\af}$ and $\breve \t \in \Omega_{\bal}$. We set $n_{\bw}=n_{\bw_1} n_{\breve \t}$. The collection $\{n_{\bw}\;|\; \bw \in \bW\}$ satisfies condition (2) in Definition \ref{def:tits}. 

Now we check condition (1). Note that $\breve S_2$ is a normal subgroup of $\breve \CT$. For any $\bw \in \bW$ and $\breve \l \in X_*(T)$, we have $n_{\bw} \breve \l(\varpi_F) n_{\bw} \i=\breve \l'(\varpi_F)$, where $\breve \l'=\bw(\breve \l) \in X_*(T)$. Let $\breve \CT'$ be the subgroup of $\breve \CT$ generated by $n_{s_{\ba}}$ for $\ba \in \breve \D$. Then any element in $\breve \CT$ is of the form $t_1 n \breve \l(\varpi_F)$ for some $t_1 \in \breve S_2$, $n \in \breve \CT'$ and $\breve \l \in X_*(T)$. If $\breve \phi(t_1 n \breve \l(\varpi_F))=1$, then $\breve \l \in \BZ \breve \Phi^\vee(G, T)$. 

Note that $\BZ \breve \Phi^\vee(G, T)$ equals to the lattice spanned by $\bw(\bb)$, where $\bw \in \bW_0$ and $\bb$ is the gradient of some $\ba \in \breve \D\backslash \breve \D_0$. By direct calculation, $n_{s_\ba} = \bb^\vee(\varpi_F) n_{s_\bb}$. In particular, $\bb^\vee(\varpi_F) \in \breve \CT'$ and hence $\breve \l(\varpi_F) \in \breve \CT'$ for all $\breve \l \in \BZ \breve \Phi^\vee(G, T)$. Therefore $\ker \breve \phi$ is contained in the subgroup generated by $\breve S_2$ and $\breve \CT'$. 

Any element of $\breve \CT'$ is of the form $n_{\bs_{i_1}}^{\pm 1} \cdots n_{\bs_{i_k}}^{\pm 1}$. Since $n_{\bs}^2 \in \breve S_2$, we have $n_{\bs_{i_1}}^{\pm 1} \cdots n_{\bs_{i_k}}^{\pm 1} \in n_{\bs_{i_1}} \cdots n_{\bs_{i_k}} \breve S_2$. 

It remains to show that

(a) If $\bs_{i_1} \cdots \bs_{i_k}=1$, then $n_{\bs_{i_1}} \cdots n_{\bs_{i_k}} \in \breve S_2$. 

We argue by induction on $k$. By the deletion condition of Coxeter groups (see \cite[Chapter IV, \S 1.3 - \S 1.5]{Bou02}), there exists $\bs_{i'_1}, \ldots, \bs_{i'_k}$ such that $n_{\bs_{i_1}} \cdots n_{\bs_{i_k}}=n_{\bs_{i'_1}} \cdots n_{\bs_{i'_k}}$ and $\bs_{i'_l}=\bs_{i'_{l+1}}$ for some $l$. We have $n_{\bs_{i'_1}} \cdots n_{\bs_{i'_k}}\in n_{\bs_{i'_1}} \cdots n_{\bs_{i'_{l-1}}} \breve S_2 n_{\bs_{i'_{l+2}}} \cdots n_{\bs_{i'_k}} =n_{\bs_{i'_1}} \cdots \hat n_{\bs_{i'_l}} \hat n_{\bs_{i'_{l+1}}} \cdots n_{\bs_{i'_k}} \breve S_2$.

Note that $\bs_{i'_1} \cdots \hat \bs_{i'_l} \hat \bs_{i'_{l+1}} \cdots \bs_{i'_k}=1$. Since there are only $k-2$ simple reflections involved, by inductive hypothesis, $n_{\bs_{i'_1}} \cdots \hat n_{\bs_{i'_l}} \hat n_{\bs_{i'_{l+1}}} \cdots n_{\bs_{i'_k}} \in \breve S_2$. Hence $n_{\bs_{i_1}} \cdots n_{\bs_{i_k}} \in \breve S_2$. 

Condition (1) of Definition \ref{def:tits} is verified. 
\end{proof}
 
\section{Tits group over $F$}\label{sec:6}

\subsection{The strategy}
The main result of this section is the existence of Tits groups over $F$ for any connected reductive group defined over $F$ and splits over $\breve F$. The strategy is as follows. 
\begin{enumerate}
    \item Let $G$ connected, reductive group over $F$ and let $\sigma$ be the Frobenius morphism on $G$ with $G(\bF)^\sigma = G(F)$. We first construct an affine pinning such that the set $\{n_{s_\ba}\;|\; \ba \in \CX\}$ is $\s$-stable for any $\s$-orbit $\CX$ of $\breve\Delta$ with $\bW_{\CX}$ finite. This result has two consequences. First, when $G$ is semisimple and simply connected, it yields a definition of the Tits group of its Iwahori-Weyl group, which is the affine Weyl group, over $F$. Second, we may construct a Tits group $\breve \CT$ over $\breve F$ that is stable under the action of a given quasi-split Frobenius morphism $\s$; 
    
    \item We then choose a suitable Frobenius morphism $\s^*$ for each inner form and show that there exists a Tits cross-section in $\breve \CT$ that is ``compatible'' with the Frobenius morphism $\s^*$; 
    
    \item Finally, we use the descent argument to show that $\breve \CT^{\s^*} \subset G(\bF)^{\s^*}$ is a Tits group of the Iwahori-Weyl group over $F$ of the group $G(\bF)^{\s^*}$. 
\end{enumerate}

\subsection{Affine pinnings and Frobenius morphisms}\label{HDSS} We have proved in Lemma \ref{TitsWaf} that given any affine pinning $\{x_\ba\;|\; \ba \in \breve\Delta\}$, the set $\{n_\bs\;|\; \bs \in \bBS\}$ satisfies the Coxeter relations, where $n_{s_\ba}= x_\ba(1)x_{-\ba}(1)x_\ba(1)$. By \S \ref{sec:WF}, for any $s \in \BS$, there exists a $\s$-orbit $\CX$ of $\breve\Delta$ with $\bW_{\CX}$ finite such that $s=\bw_{\CX} \in \bW$. In this section, we show the following.

\begin{proposition}\label{sigmarepsWaf}
Let $G$ be a connected reductive group defined over $F$ that splits over $\bF$. Let $\s$ be a Frobenius morphism on $G_\bF$.  There exists an affine pinning $\{x_\ba\;|\; \ba \in \breve\Delta\}$ such that the set $\{n_{s_\ba}\;|\; \ba \in \CX\}$ is $\s$-stable for any $\s$-orbit $\CX$ of $\breve\Delta$ with $\bW_{\CX}$ finite. 
\end{proposition}

\begin{remark}
(1) If $\sigma$ is a quasi-split Frobenius, that is, if $G_\bF^\sigma$ is quasi-split, then $\bW_{\CX}$ is finite for any $\sigma$-orbit $\CX$ in $\breve\Delta$. \\
(2) Assuming that $G$ is absolutely simple, the finiteness assumption on  $\bW_\CX$ fails only for inner forms of type $A$, that is, if $G_\bF$ is split of type $A$, and if $\sigma$ is such that $G_\bF^\sigma$ is a group over $F$ whose adjoint group is $\PGL_1(D)$ for a suitable division algebra $D$. Such a group is anisotropic over $F$ with trivial affine Weyl group and its Iwahori-Weyl group has only length zero elements.\\ (3) Recall from \S \ref{sec:WF} that the elements of $\BS$ are in bijection with $\sigma$-orbits $\CX$ such that $\bW_{\CX}$ is finite. So, it suffices to consider such orbits to construct the Tits group over $F$. However, while the proof below uses the assumption that the $\s$-orbit $\CX$ is such that $\bW_{\CX}$ is finite, we will show in Proposition \ref{prop:power-s} through a different argument that the finiteness assumption on $\bW_{\CX}$ can be dropped.
\end{remark}

\begin{proof}
Let $\{x_\ba\;|\; \ba \in \breve\Delta\}$ be an affine pinning and $\CX$ be a $\s$-orbit in $\breve\Delta$ such that $\bW_\CX$ is finite. Let $k = \# \CX$.  

Fix $\ba \in \CX$ and let $\bb$ be the gradient of $\ba$. Since $W_{\CX}$ is finite, we have $\bw_{\CX} \in \bW_{\CX}^\s$. In particular, $W_\CX^\sigma \neq 1$. Thus $\bb|_A \neq 0$ and hence is a root $b$ in $\Phi(G,A)$. We show that 

(a) There exists $u \in \fkO_{\bF}^\times$ such that $\sigma^k(x_\ba(u)) = x_\ba(u)$. 

Let $v \in \CA(A,F) \subset \CA(T, \bF)$ and $r \in \BR$. For $\bb \in \breve\Phi(G,T)$, let $U_\bb(\bF)_{v,r} \subset U_\bb(\bF)$ denote the filtration of root subgroup $U_{\bb}(\bF)$ as in \cite[\S 4.3]{BT2}. We recall the definition of the filtration of the root subgroup $U_b(F)$ (cf. \cite[\S 5.1.16 - 5.1.18]{BT2}). Let $\breve\Phi^b: = \{ \bc \in \breve\Phi(G,T)\;|\; \bc|_A = b \text{ or } 2b\}$. This is a $\s$-stable positively closed subset of $\breve\Phi(G,T)$; that is if $\bc_1.\bc_2 \in \breve\Phi^b$ such that $\bc+\bc'$ is a root, then $\bc+\bc' \in \breve\Phi^b$. For any fixed ordering, the subset
\begin{align}\label{ubr}
    U_b(\bF)_{v,r} := \prod_{\bc \in \breve\Phi^b, \bc|_A = b} U_\bc(\bF)_{v,r}  \prod_{\bc \in \breve\Phi^b_{nd}, \bc|_A = 2b} U_\bc(\bF)_{v,2r}
\end{align}
is a subgroup of $U_b(\bF)$. Let $U_b(F)_{v,r} := U_b(\bF)_{v,r} \cap U_b(F)$. 

We have
\begin{align}\label{ubvv0}
    U_\bb(\bF)_{v,r} = U_{\bb}(\bF)_{\bv_0, r+ \bb(v - \bv_0)},
\end{align}
where $\bv_0$ is the special vertex in $\CA(S,\bF)$ fixed in \S \ref{notation}. 

The pinning $x_\ba:  \BG_a \rightarrow U_{\bb}(\bF)$ satisfies $x_{\ba}(\fkO_{\bF}) = U_\bb(\bF)_{\bv_0,r_0}$ and  $x_{\ba}(\fkp_{\bF}) = U_{\bb}(\bF)_{\bv_0,s_0}$ for a suitable $r_0<s_0$. 

Using \eqref{ubvv0} and then adjusting $r$ if necessary, we ensure 
\begin{align}\label{xabf}
x_{\ba}(\fkO_{\bF}) = U_{\bb}(\bF)_{v,r},\;\; x_{\ba}(\fkp_{\bF}) = U_{\bb}(\bF)_{ v, r+}
\end{align}
for a suitable $r \in \BR$. In particular, $U_{\bb}(\bF)_{v,r}\neq U_{\bb}(\bF)_{v,r+}$.
Now \cite[Proposition 5.1.19]{BT2} implies  that $U_b(F)_{v,r} \neq U_b(F)_{v,r+}$. In other words, there exists $u' \in U_b(F)_{v,r}$ such that
\[ u' \notin \prod_{\bc \in \breve\Phi^b, \bc|_A = b} U_\bc(\bF)_{v,r+}  \prod_{\bc \in \breve\Phi^b_{nd}, \bc|_A = 2b} U_\bc(\bF)_{v,2r}.\]
Note that $\sigma^k$ fixes every element of $\CX$ (and hence also every element of $\Phi^b$) and $\sigma^k(u') = u'$. By \eqref{ubr}, we have
\[u' \in \prod_{\bc \in \breve\Phi^b, \bc|_A = b} U_\bc(\bF)_{v,r}^{\sigma^k}  \prod_{\bc \in \breve\Phi^b_{nd}, \bc|_A = 2b} U_\bc(\bF)_{v,2r}^{\sigma^k},\]
and 
\[u' \notin \prod_{\bc \in \breve\Phi^b, \bc|_A = b} U_\bc(\bF)_{v,r+}^{\sigma^k} \prod_{\bc \in \breve\Phi^b_{nd}, \bc|_A = 2b} U_\bc(\bF)_{v,2r}^{\sigma^k}.\]

Thus there exists $\bc \in \Phi^b, \bc|_A = b$ such that
$ U_\bc(\bF)_{v,r+}^{\sigma^k}  \subsetneq U_\bc(\bF)_{v,r}^{\sigma^k}$. Since $\bc = \sigma^i(\bb)$ for a suitable $i<k$, we also have
\[U_\bb(\bF)_{v,r+}^{\sigma^k}  \subsetneq U_\bb(\bF)_{v,r}^{\sigma^k}.\]

Let $u'' \in U_\bb(\bF)_{v,r}^{\sigma^k}\backslash U_\bb(\bF)_{v,r+}^{\sigma^k}$ and $u= x_\ba^{-1}(u'')$. Then $\sigma^k(x_\ba(u)) = x_\ba(u)$. By \eqref{xabf}, $u\in \fkO_{\bF}^\times$.

(a) is proved. 

Since $x_\ba(u) x_{-\ba}(u^{-1}) x_\ba(u) \in N_G(T)(\bF)$, we have $$x_\ba(u) \sigma^k(x_{-\ba}(u^{-1}))x_\ba(u)=\s^k(x_\ba(u)) \sigma^k(x_{-\ba}(u^{-1})) \s^k(x_\ba(u))\in N_G(S)(\bF).$$ The uniqueness assertion in \cite[\S 6.1.2, (2)] {BT2} implies that $\sigma^k(x_{-\ba}(u^{-1})) = x_{-\ba}(u^{-1})$.

Let $x_\ba' = x_\ba \circ m_u$, where $m_u$ is the multiplication by $u$. We consider the pinning $\{x_\ba', \sigma \circ x_\ba', \cdots \sigma^{k-1} \circ x_{\ba}'\}$. Then $x_{\ba}'(1) =x_{\ba}(u)$ and $x_{-\ba}'(1) = x_{-\ba}(u^{-1})$. For $\bc \in \CX$, let $n_{s_\bc}' = x_{\bc}'(1)x_{-\bc}'(1)x_{\bc}'(1)$ be the representative in $N_G(S)(\bF)$ of $s_\bc$ obtained using this pinning. Then $\sigma ^{k}(n_{s_\bc}' )= n_{s_\bc}'$ and the set $\{n_{s_\bc}'\;|\; \bc \in \CX\} = \{n_{s_\ba}', \sigma(n_{s_\ba}'), \cdots, \sigma^{k-1}(n_{s_\ba}')\}$ is $\sigma$-stable.
\end{proof}

\subsection{The Frobenius morphism for each inner form}\label{FMIF} In this rest of this section, let $G$ denote a connected, reductive group over $F$ that is quasi-split over $F$ and split over $\bF$.  Let $\sigma$  denote the Frobenius morphism on $G_\bF$ so that the $F$-structure it yields is $G$. We will later construct for each $F$-isomorphism class of inner twists of $G$  a suitable Frobenius morphism $\sigma^*$ and let $G^* = G_\bF^{\sigma^*}$ be the $F$-group in the given isomorphism class of inner twists. 

By Proposition \ref{sigmarepsWaf}, there exists an affine pinning $\{x_\ba\;|\; \ba \in \breve\Delta\}$ such that the set $\{n_{\bs}\;|\; \bs \in \bBS\}$ is $\s$-stable. For $\blambda \in X_*(T)$, let $n_\blambda = \blambda(\varpi_F)$. Then $\s(n_{\blambda})=n_{\s(\blambda)}$. Let $\breve \CT$ be the Tits group of $\bW$ generated by $\breve S_2$, $\{n_{\bs}\;|\; \bs \in \bBS\}$ and $\{n_{\blambda}\;|\;\blambda \in X_*(T)\}$. Then $\breve \CT$ is stable under the action of $\s$. 

We will choose a suitable Frobenius morphism $\s^*$ for each $F$-isomorphism class of inner twists of $G$ such that $\breve \CT$ is stable under the action of $\s^*$. Finally, we will show that $\breve \CT^{\s^*}$ is a Tits group over $F$ for $G^*$. 

\subsubsection{The group $(\Omega_{\bal,\ad})_{\sigma}$}\label{omega-1} 

The $F$-isomorphism classes of inner twists of $G$ is parametrized by $H^1(\langle \sigma \rangle, G_{\ad}(\bF))$. 
By \cite[Lemma 2.1.2 and \S 2.3-2.4]{DR}, we have
\[H^1(\langle\sigma\rangle, G_{\ad}(\bF)) \cong H^1(\langle\sigma\rangle, \Omega_{\bal, \ad})=(\Omega_{\bal,\ad})_{\sigma}.\]

Now, we describe the group $\Omega_{\bal,\ad}$ in more detail. We may assume that $G_{\bF, \ad}$ is $\bF$-simple.  We will use the same labeling of the roots in $\Phi(G,T)$ as in \cite{Bou02} and we denote the indexing set of simple reflections by $I$. With this, the set $\{s_\ba\;|\; \ba \in \breve\Delta_0\}$ is identified with $\{\bs_i\;|\; i \in I\}$.  Note that we have used the letter $I$ for the indexing set for the simple reflections; this should not cause any confusion, since the Iwahori subgroup will not be mentioned in the rest of this paper. For $J \subset I$, let $\by_J$ be the maximal element in the subgroup generated by $\bs_i, i \in J$. 

Let $\breve\rho^\vee$ be the half-sum of the positive coroots in any positive system. Let $i \in I$. If $\omega_i$ is minuscule, we denote by $\bnu_{\ad, (i)}=t_{\omega^\vee_i} \by_{(i)} \in \Omega_{\bal, \ad}$ the corresponding element. Here $\by_{(i)}=\by_{I \backslash \{i\}} \by_I$. Note that if $\bnu_{\ad, (i)}=\bnu_{\ad, (j)}^k$ for some $k \in \BN$, then we also have that $\by_{(i)}=\by_{(j)}^k$. 

The description of $\Omega_{\bal,ad}$ is given in the following table. We list according to the type of the local Dynkin diagram of $G_{\bF,\ad}$. We only list the types for which $\Omega_{\bal,\ad}$ is non-trivial. In the last column, we make a choice of generator $\nuag$ in the case where $\Omega_{\bal, \ad}$ is cyclic. Such element $\nuag$ will be used later.
\begin{table}[H]
  \begin{center}
    \caption{The group $\Omega_{\bal,\ad}$}
    \label{tab:tableOmega}
    \begin{tabular}{|m{1.5cm}|m{1.6cm}|m{2cm}|m{4cm}|m{1.7cm}|} 
    \hline 
    \vspace{5pt}
     Type &$\Phi(G,T)$  & $ \Omega_{\bal,\ad}$ & Elements & Generator\\[5pt]
    \hline
    \vspace{5pt}
      $ A_n$&$ A_n$ &$ \BZ/(n+1)\BZ$ &$\{1,\bnu_{\ad,(i)}\; |\; 1 \le i \le n\}$& $\bnu_{\ad,(1)}$\\[5pt]
     \hline
    \vspace{5pt}
    $ B_n$&$ B_n$ &$ \BZ/2\BZ$ &$\{1,\bnu_{\ad,(1)}\}$& $\bnu_{\ad,(1)}$\\[5pt]
    \hline
    \vspace{5pt}
    $ C_n$&$ C_n$ &$ \BZ/2\BZ$ &$\{1,\bnu_{\ad,(n)}\}$& $\bnu_{\ad,(n)}$\\[5pt]
    \hline
    \vspace{5pt}
    $ D_n$ &$ D_n, 2 \mid n$ &$ \BZ/2\BZ \times \BZ/2\BZ $ &$\{1,\bnu_{\ad,(1), \bnu_{\ad,(n-1)}, \bnu_{\ad,(n)}}\}$& N/A \\[5pt]
    \hline
    \vspace{5pt}
    $ D_n$ &$ D_n, 2 \nmid n$ &$ \BZ/4\BZ$ &$\{1,\bnu_{\ad,(1), \bnu_{\ad,(n-1)}, \bnu_{\ad,(n)}}\}$& $\bnu_{\ad,(n)}$\\[5pt]
    \hline
    \vspace{5pt}
    $ E_6$ &$ E_6$ &$ \BZ/3\BZ$ &$\{1,\bnu_{\ad,(1)}, \bnu_{\ad,(6)}\}$& $\bnu_{\ad,(1)}$\\[5pt]
    \hline
    \vspace{5pt}
    $E_7$ &$ E_7$&$ \BZ/2\BZ$ &$\{1,\bnu_{\ad,(7)}\}$& $\bnu_{\ad,(7)}$\\[5pt]
    \hline
    
    \end{tabular}
  \end{center}
\end{table}

If $\s$ acts trivially on $\Omega_{\bal,\ad}$, then $(\Omega_{\bal,\ad})_{\sigma} \cong \Omega_{\bal,\ad}$. If the action of $\s$ on $\Omega_{\bal,\ad}$ is nontrivial, then $$(\Omega_{\bal,\ad})_{\sigma}=\begin{cases} \BZ/2 \BZ, & \text{ if } G_{\bF,\ad} \text{ is of type } A_{2n+1} \text{ or } D_n; \\ 1, & \text{ otherwise}. \end{cases}$$

\subsubsection{The construction of suitable Frobenius morphism $\s^*$}\label{Conssigma*}  

Let $j: G_\bF \rightarrow G_{\bF, \ad}$ denote the adjoint quotient. This induces maps $T \rightarrow T_{\ad}$, $\bW \rightarrow \bW_{\ad}$ and $\Omega_{\bal} \rightarrow \Omega_{\bal,\ad}$, and we will denote all these maps by $j$ as well.  
The exact sequence $1 \rightarrow Z \rightarrow T \rightarrow T_{\ad} \rightarrow 1$ induces exact sequences
\[1 \rightarrow X_*(Z^0) \rightarrow X_*(T) \xrightarrow{j} X_*(T_{\ad})\]
and 
\[1 \rightarrow X_*(Z^0) \rightarrow \Omega_\bal \xrightarrow{j} \Omega_{\bal,\ad},\]
where $Z$ is the center of $G$ and $Z^0$ is the maximal torus in the center of $G$.

We will construct a suitable Frobenius morphism $\s^*$ associated to each $F$-isomorphism class of inner twists of $G$. It suffices to consider the case where $G_{F, \ad}$ is $F$-simple. 

We first discuss the case where $G_{\bF, \ad}$ is $\bF$-simple. 

We choose as follows the element $\nua$ in $\Omega_{\bal, \ad}$ whose image in  $(\Omega_{\bal,\ad})_\s$ the parametrizes the inner twist $G^*$ of $G$. If $G_{\bF, \ad}$ is of type $A_{2n+1}$ or $D_{2n+1}$ for some $n \in \BN$ and the $\s$-action on $\Omega_{\bal, \ad}$ is nontrivial, then $\Omega_{\bal,\ad}$ is nontrivial, then $(\Omega_{\bal,\ad})_{\sigma}=\BZ/2 \BZ$. In this case, we take $\nua=1$ if $G^*$ is quasi-split over $F$ and $\nua=\nuag$ if $G^*$ is not quasi-split over $F$. Here $\nuag$ is the generator of $\Omega_{\bal, \ad}$ listed in Table \ref{tab:tableOmega}. In other cases, we may take $\nua$ to be any element in $\Omega_{\bal, \ad}$ that corresponds to the $F$-isomorphism class of $G^*$. The choice of $\nua$ is not essential, but will simplify some calculations in the rest of this section. 

Let $\nua =t_\bea \bz$. We construct suitable liftings of $t_{\bea}$ and $\bz$. 

The lifting of $t_{\bea}$ is constructed as follows. Note that the quotient $X_*(T_\ad)/j(X_*(T))$ is finite.  Consider the element $t_{\bea} \in X_*(T_{\ad})$. Let $k\geq 1$ be the smallest integer such that $k\bea \in j(X_*(T))$. Write
\[k\bea = j(\be) \text { for some } \be \in X_*(T).\]

Let $ \bnu = t_\be \bz.$ Note that $\bnu \in \bW$, but need not lie in $\Omega_{\bal}$. We know that $n_\bea =\bea(\varpi_F)$.  Set 
\begin{align}\label{gbedef}
    g_\be = \be(\varpi_F^{1/k}).
\end{align}
Note that $g_\be \in T(\bar F) \subset G(\bar F)$ and $j(g_\be) = n_\bea.$ Note that for each root $\ba \in \breve\Phi(G,T)$, we have $\frac{\langle \ba, \be \rangle}{k} \in \BZ$ because $j(\ba) = \ba$ and $\langle \ba, \be \rangle = \langle j(\ba), j(\be) \rangle = k \langle \ba, \bea \rangle \in k\BZ$. In particular, conjugation by $g_\be$ preserves $G(\breve F)$. 

Now we construct a lifting $g_{\bz}$ of $\bz$ in $\breve \CT$. If $G_{\bF,\ad}$ is of type $D_n$ with $n$ even, then we set $g_{\bz}=n_{\bz}$. Otherwise, $\Omega_{\bal, \ad}$ is a cyclic group. From our construction, $\nua =\nuag^i$ for $0 \le i<|(\Omega_{\bal, \ad})_\s|$. We then have $\bz=\bzg^i$. Set $g_{\bz}=n_{\bzg}^i$.  Let $g_\bnu= g_\be g_\bz \in G(\bar F)$. We set $$\s^*=\Ad(g_{\bnu}) \circ \s.$$ 

Next, suppose $G_{\bF,\ad}$ is not simple. 

By our assumption $G_{F, \ad}$ is simple. We may write $G_{\ad} =\Res_{L_k/F} G_{\ad}'$, where $L_k$ is a finite unramified extension of $F$ of degree $k$ contained in $\bF$ and $G'_{\bF, \ad}$ is $\bF$-simple. Then 
\begin{equation}\label{adjnonsimple}
    G_{\bF, \ad} =G_{\bF, \ad}^{(1)} \times \cdots \times G_{\bF, \ad}\ik,
\end{equation} where $G_{\bF, \ad}^{(1)} \cong \cdots \cong G_{\bF, \ad}\ik \cong G'_{\bF, \ad}$ and the action of $\s$ permutes transitively the simple factors $G_{\bF, \ad}^{(1)}, \ldots, G_{\bF, \ad}\ik$. We may also write $\Omega_{\bal, \ad} = \Omega_{\bal, \ad}^{(1)} \times \cdots \times \Omega_{\bal, \ad}\ik$, where $\Omega_{\bal, \ad}^{(1)} \cong \cdots \cong \Omega_{\bal,\ad}\ik $ are as in Table \ref{tab:tableOmega}. Then the projection map $\Omega_{\bal, \ad}^{(1)} \to (\Omega_{\bal, \ad})_{\sigma}$ is surjective. In fact, $(\Omega_{\bal, \ad}^{(1)})_{\sigma^k} \xrightarrow{\cong} (\Omega_{\bal, \ad})_{\sigma}$.  Let $\nua = t_\bea \bz \in \Omega_{\bal, \ad}^{(1)}$ such that its image in $(\Omega_{\bal, \ad})_{\sigma}$ parametrizes the isomorphism class of $G^*$. We construct $g_{\bnu} \in G(\bar F)$ as above. More precisely, write $\nua = t_\bea \bz$. Then $t_\be$ and $g_\be$ have been constructed in \eqref{gbedef}. If $G_{\bF,\ad}^{(1)}$ is of type $D_n$ with $n$ even, then we set $g_{\bz}=n_{\bz}$. Otherwise we have $\nua=\nuag^i$ for $0 \le i<|\left(\Omega_{\bal, \ad}^{(1)}\right)_{\s^k}|$.  We also have $\bz=\bzg^i$. Set $g_{\bz}=n_{\bzg}^i$. Let $g_\bnu = g_\be g_\bz$ and let $\sigma^* = Ad(g_\bnu) \circ \sigma$.

\subsubsection{The action of $\s^*$ on $\breve \CT$}\label{identitiestp} It is easy to see that $\breve S_2$ is stable under the action of $\s^*$. For each $\blambda \in X_*(T)$, we have 
\begin{equation}\label{lem1:trans} 
\sigma^*(n_\blambda) = \Ad(g_\bz)(n_{\sigma(\lambda)}) = n_{\bz(\sigma(\lambda))} = n_{\sigma^*(\blambda)}.
\end{equation}

Note that $\sigma^*$ acts as $\Ad(\bea) \circ \sigma$ on $\bW_{\ad}$. So for any $\by \in \bW_0$, $\sigma^*(\by) = t_{\bea - \by'(\bea)} \by'$, where $\by' = \Ad(\bz)(\sigma(\by))$. Note that $\bea - \by'(\bea)\in \BZ\Phi^\vee(G,T)$.

Since $\breve \CT$ is generated by $\breve S_2$, $m(\by)$ for $\by \in \bW_0$ and $n_{\blambda}$ for $\blambda \in X_*(T)$, by the following lemma, we have $$\s^*(\breve \CT)=\breve \CT.$$ 

\begin{lemma}\label{lem2:trans} Let $\by \in \bW_0$. Then for any lifting $m(\by) \in \breve\CT$, we have \[\sigma^*(m(\by)) = n_{\bea - \by'(\bea)} g_\bz \sigma(m(\by))g_\bz^{-1}.\] 
\end{lemma}
\begin{proof} We have $\Ad(\bnu)(\sigma(\by)) = t_{\be - \by'(\be)} \bz \sigma(\by) \bz^{-1}$. Now 
\begin{align*} j(\be -\by'(\be)) &=j(\be)-j(\by'(\be))=k(\bea)-\by'(j(\be))=k(\bea)-k \by'(\bea) \\
&=k(\bea - \by'(\bea)) = k(j(\bea - \by'(\bea))).
\end{align*}
Here the last equality follows from the fact that $\bea - \by'(\bea) \in \BZ\breve\Phi^\vee(G,T)$ and the restriction of the map $j$ to $\BZ\breve\Phi^\vee(G,T)$ (which is just $X_*(T_{sc})$) is the identity map. 
Hence there exists $\bmu \in X_*(Z^0)$ such that 
\[\be - \by'(\be) = k(\bea - \by'(\bea)) +\bmu.\] 

Since $\bmu \in X_*(Z^0)$, $\by'(\bmu)=\bmu$ and thus 
\[(\by')^{i}(\be) - (\by')^{i+1}(\be) = k((\by')^{i}(\bea) - (\by')^{i+1}(\bea))+\bmu\] for any $i$. Let $l$ be the order of $\by'$. Then 
\begin{align*}
l \bmu=\sum_{i=0}^{l-1} \bigl(k((\by')^{i}(\bea) - (\by')^{i+1}(\bea))+\bmu\bigr)=\sum_{i=0}^{l-1} \bigl((\by')^{i}(\be) - (\by')^{i+1}(\be)\bigr)=0.
\end{align*}

So $\bmu=0$ and $\be - \by'(\be) = k(\bea - \by'(\bea))$. Then
\begin{align*}
    \sigma^*(m(\by)) = g_\be \Ad(\by')(g_{\be}\i) g_\bz \sigma(m(\by))g_{\bz}^{-1}. 
\end{align*}
It remains to show that $g_\be \Ad(\by')(g_{\be}\i) =n_{\bea - \by'(\bea)}$. 

With the definition of $g_\be$ in \eqref{gbedef}, we have \begin{align*}
    g_\be \Ad(\by')(g_{\be}\i) &= (\be - \by'(\be))(\varpi_F^{1/k})= (\tilde\eta_{\ad} - \by'(\tilde\eta_\ad))(\varpi_F) = n_{\bea - \by'(\bea)}.
\end{align*}
This finishes the proof of the lemma. 
\end{proof}

Now we state the main result of this section. 

\begin{theorem}\label{thm:tits-F}
Let $G$ be a connected reductive group, quasi-split over $F$ and split over $\breve F$. Let $\s^*$ be the Frobenius morphism associated to a given $F$-isomorphism class of inner twists of $G$. Then $\breve \CT^{\s^*}$ is a Tits group of the Iwahori-Weyl group of $G(\breve F)^{\s^*}$. 
\end{theorem}

In the rest of this section, we will prove this theorem. The proof involves, among other things, some identities on the finite Tits groups, which we now summarize.

\subsection{Some identities in finite Tits groups} In this subsection, let $\mathfrak{F}$ be any field and let $G$ be a split, connected, reductive group over $\mathfrak{F}$. Let $\CT_{\fin}$ be the Tits group of the absolute Weyl group $W_0$ of $G(\mathfrak{F})$ and $\{n_w\}_{w \in W_0}$ is a Tits cross-section of $W$ on $\CT_{\fin}$. In the application to the proof of Theorem \ref{thm:tits-F}, $\mathfrak{F}$ is the field $\breve F$ and $\CT_{\fin}$ is the subgroup of $\breve \CT$ generated by $n_{\bs}$ for $\bs \in \breve W_0$. However, the identities on the finite Tits group hold in the general setting. 

Let $\{s_i\;|\; i \in I\}$ be the set of simple reflections of the absolute Weyl group and $\{a^\vee_i\;|\;i \in I\}$ be the set of simple coroots. Then $n_{s_i}^2=a^\vee_i(-1)$. For any subset $J \subset I$, let $\rho^\vee_J$ be the half sum of positive coroots spanned by $\{a_j^\vee\;|\; j \in J\}$ and $y_J$ be the maximal element in the subgroup generated by $s_i$ for $i \in J$. We will simply write $\rho^\vee$ for $\rho^\vee_{I}$. For any $i \in I$, we set $$y_{(i)}=y_{I\backslash \{i\}} y_{I}.$$ We are interested in the power of $n_{y_{(i)}}$ when $\o^\vee_i$ is a minuscule coweight. 
This is calculated using the following result.  

\begin{proposition}[Proposition 3.2.1 of \cite{Ro}]\label{Ros}
Let $u, v \in W_0$. Then $$n_u n_v=n_{u v} \; \Pi_{a>0, v(a)<0, u v(a)>0} a^\vee(-1).$$
\end{proposition}

The following corollary is an easy consequence of the proposition above and some results in \cite[\S 3]{Ad}. 
\begin{corollary}\label{cor:FTG}
\begin{enumerate}[(1)]
\item Suppose that $G_{\mathfrak F, \ad}$ is $\mathfrak F$-simple. Let $\o^\vee_i$ be a minuscule coweight and $k$ be the order of $y_{(i)}$ in $W_0$. Then $n_{y_{(i)}}^k$ is the center of $G(\mathfrak F)$. 
\item Suppose that $G$ is of type $A_n$. For $0 \le i \le n$, 
    \[n_{y_{(1)}}^{i+1} = \begin{cases} n_{y_{(1)}^{i+1}}, & \text{ if } i \text{ is even},\\
    (a_1 ^\vee + a_3^{\vee} + \cdots + a_{i}^\vee)(-1) n_{y_{(1)}^{i+1}}, & \text{ if } i \text{ is odd}.
    \end{cases} \]
\item Suppose that $G$ is of type $D_n$ with $n$ odd. Then 

(a) \[n_{y_{(n)}}^2 = \begin{cases} 
    (a_2^\vee +a_4^\vee + \cdots +a_{n-1}^\vee)(-1) n_{y_{(1)}}, & \text{ if } n \equiv 1 \mod 4, \\ 
    (a_2^\vee +a_4^\vee + \cdots +a_{n-3}^\vee+a_n^\vee)(-1) n_{y_{(1)}}, & \text{ if } n \equiv 3 \mod 4.
\end{cases} \]

(b) \[n_{y_{(n)}}^3 = \begin{cases} 
    (a_{n-1}^\vee+a_n^\vee)(-1) n_{y_{(n-1)}}, & \text{ if } n \equiv 1 \mod 4, \\ 
    n_{y_{(n-1)}}, & \text{ if } n \equiv 3 \mod 4.
\end{cases} \]

(c) $n_{y_{(n)}}^4=(a^\vee_{n-1}+a^\vee_n)(-1)$.

\item Suppose that $G_{\mathfrak F, \ad}$ is of type $D_{2n}$ and $\{i, j, k\}=\{1, 2n-1, 2n\}$. Then there exists a central element $z$ of $G(\mathfrak F)$ such that \[n_{y_{(i)}}n_{y_{(j)}} = z\cdot n_{y_{(k)}}= n_{y_{(j)}}n_{y_{(i)}}.\]
\end{enumerate}
\end{corollary}
\begin{proof}
All the parts of the corollary are consequences of Proposition \ref{Ros} and some explicit calculations. In the case where $G$ is almost simple over $\mathfrak F$, Adrian showed in \cite[Proposition 3.3]{Ad} that $n_{y_{(i)}}^k=1$, where $k$ is the order of $y_{(i)}$ in $W_0$. This implies (1). Parts (2), (3) are direct consequences of Proposition \ref{Ros}. Part (4) is deduced from \cite[Proof of Theorem 3.5]{Ad}.
\end{proof}

\subsection{The $\s^*$-stable liftings of $\bBS$} \label{verHDS} 

In this subsection, we will prove the following result.

\begin{proposition}\label{prop:power-s}
Let $\bs\in \bBS_0$ and $\CX$ be the $\s^*$-orbit of $\bs$. Then 
$$(\s^*)^{|\CX|}(n_\bs)=n_\bs.$$
\end{proposition}

As a consequence, we obtain the following stronger version of Proposition \ref{TitsWaf}. 

\begin{corollary}\label{cor:m(s)}
There exists a set of representatives $\{m(\bs)\;|\; {\bs \in \bBS}\}$ in $\breve \CT$ that is $\sigma^*$-stable. 
\end{corollary}

\begin{proof}
It suffices to consider the case where $G_{F, \ad}$ is $F$-simple. In this case, $\s^*$ acts transitively on the set of connected components of the affine Dynkin diagram of $G_F$. The case $\s^*=\s$ is already proved. Now we assume that $\s^* \neq \s$. Then each $\s^*$-orbit on $\bBS$ contains a simple reflection in $\bBS_0$. For each $\s^*$-orbit $\CX$, we fix a representative $\bs_\CX$ such that $\bs_\CX \in \CX \cap \bBS_0$. Then any element $\bs \in \bBS$ is of the form $\bs=(\s^*)^l(\bs_\CX)$ for a unique $\s^*$-orbit $\CX$ and a unique $l$ with $0 \le l<|\CX|$. We then set $m(\bs)=(\s^*)^l n_{\bs_{\CX}} \in \breve \CT$. By Proposition \ref{prop:power-s}, $\{m(\bs)\;|\; {\bs \in \bBS}\}$ is $\s^*$-stable. 
\end{proof}

\subsubsection{Reduction step}
We first explain how to reduce ourselves to the case when $G_{\bF, \ad}$ is $\bF$-simple. We keep notations as in \S\ref{Conssigma*}. Note that
\[\bW_{\af} =\bW_{\af}^{(1)} \times \bW_{\af}^{(2)}  \cdots \times \bW_{\af}^{(k)},\]
with $\bW_{\af}^{(1)}  \cong \bW_{\af}^{(2)}  \cong \cdots \cong \bW_{\af}^{(k)} $ and $\sigma$ permutes these factors transitively. Write $\bBS = \bBS^{(1)} \times \cdots \times \bBS^{(k)}$. The element $\bs = (\bs^{(1)}, \bs^{(2)}\cdots, \bs^{(k)})$ and $n_{\bs} = n_{\bs^{(1)}}\cdot n_{\bs^{(2)}} \cdots \cdot n_{\bs^{(k)}}$. Note that $\sigma^k$ stabilizes the set $\bBS^{(1)}$ and that $(\sigma^*)^k$ acts as $\Ad(\nua) \circ \sigma^k$ on $\bBS^{(1)}$. Let $\CX^{(1)}$ be the $(\Ad(\nua) \circ \sigma^k)$-orbit of $\bs^{(1)}$ in $\bBS^{(1)}$. Note that $|\CX| = k|\CX^{(1)}|$.  Thus $(\sigma^*)^{|\CX|}(n_{\bs}) = n_{\bs}$ if and only if $(\Ad(g_\bnu) \circ \sigma^k)^{|\CX^{(1)}|}(n_{\bs^{(1)}}) = n_{\bs^{(1)}}$. In particular, we may reduce ourselves to the case when $k=1$, i.e., the case when $G_{\bF, \ad}$ is $\bF$-simple. In this case, $\Omega_{\bal, \ad}$ is as in Table \ref{tab:tableOmega} and we drop all the superscripts in the rest of the argument. 

Next we show that it suffices to prove the equality \eqref{SCX} below. This equality only involves the elements from the finite Tits group. 

If $\nua=1$, then $\s^*=\s$ and (a) follows from the fact that the set $\{n_{\bs}\;|\; \bs \in \bBS\}$ is $\s$-stable. 

Now we assume that $\nua \neq 1$. Recall that $\nua =t_\bea \bz\in \Omega_{\bal,\ad}$. We have $(\nua \circ \s)^{|\CX|}=t_{\breve \xi_{\ad}} (\bz \circ \s)^{|\CX|} \in \bW_{\ad} \rtimes \<\s\>$ for some $\breve \xi_{\ad}\in X_*(T_{\ad})$. Since $(\s^*)^{|\CX|}(\bs)=\bs$, we have $(\Ad(\bz) \circ \s)^{|\CX|}(\bs)=\bs$ and $\bs(\breve \xi_{\ad})=\breve \xi_{\ad}$.
Recall in \S\ref{Conssigma*}, we have $k \breve \eta_{\ad}=j(\breve \eta)$ for some $\breve \eta \in X_*(T)$. Since $\breve \xi_{\ad}$ is an integral linear combination of the $\bW_0$-orbit of $\breve \eta_{\ad}$, we have $k \breve \xi_{\ad}=j(\breve \xi)$ for some $\breve \xi \in X_*(T)$. By \eqref{lem1:trans} and Lemma \ref{lem2:trans}, we have $$(\s^*)^{|\CX|}(n_\bs)=\Ad(\breve \xi(\varpi_F^{1/k})) (\Ad(g_{\bz}) \circ \s)^{|\CX|}(n_\bs).$$
By the proof of Lemma \ref{lem2:trans}, we have $\Ad(\breve \xi(\varpi_F^{1/k})(n_\bs)=n_\bs$. Thus it remains to prove that
\begin{align}\label{SCX}
    (\Ad(g_{\bz}) \circ \s)^{|\CX|}(n_\bs)=n_\bs.
\end{align}

\subsubsection{The case where $\s^*(\bs)=\bs$}\label{orbit1}
If $G_{\bF, \ad}$ is of type $D_{2n}$, $g_\bz = n_\bz$. Otherwise $g_\bz = n_{\bz_0}^i$ for a suitable $0 \le i <|(\Omega_{\bal, \ad})_{\sigma}|$. For type $D_{2n}$ and for the other types with $i=1$, the statement follows from Coxeter relations. In more detail, since $\sigma^*(\bs) = \bs$  we know that $\sigma^*(\bb) = \pm \bb$.  Since $\sigma^*$ preserves $\breve\Delta$, we see that $\sigma^*(\bb) = \bb$. But $\sigma^*(\bb) = \bz(\sigma(\bb))$. Now, since $\sigma(\bb) \in \breve\Delta_0$ and $\bz(\sigma(\bb)) \in \breve\Delta_0$, we see that  $l(\bz s_{\sigma(\bb)}) = l(\bz)+1 = l(s_\bb \bz)$. By Condition (2)(b)$^\dagger$ of \S \ref{sec:tits}, $g_{\bz}n_{s_{\sigma(\bb)}} = n_{\bz s_{\sigma(\bb)}} =n_{s_\bb \bz}= n_{s_\bb}g_{\bz}$. 

If $G_{\bF, \ad}$ is not of type $D_{2n}$ and $i>1$, then the action of $\s$ on $\Omega_{\bal, \ad}$ is trivial and $G_{\bF, \ad}$ is of type $A_n$, $D_{2n+1}$ or $E_6$. We have $\Ad(\bz_0^i)(\bs) = \bs$ for $1<i <|\Omega_{\bal, \ad}|$ and we need to prove that $n_{\bz_0}^i n_\bs n_{\bz_0}^{-i} = n_\bs$. 

If $G_{\bF, \ad}$ is of type $A_n$, then since $\bz_0^i$ does not have any fixed points on $\bBS$, there is no such $\bs$ and the statement is trivial. 

If $G_{\bF, \ad}$ is of type $E_6$, then $|\Omega_{\bal, \ad}|=3$. If $i>1$, then $i=2$ and $\bz_0=(\bz_0^2)^2$. If $\bz_0^2$ fixes $\bs$, then necessarily $\bz_0$ fixes $\bs$. By Condition (2)(b)$^\dagger$ of \S \ref{sec:tits}, $n_{\bz_0} n_{\bs} = n_{\bs} n_{\bz_0}$. Hence $g_\bz n_\bs g_{\bz}\i = n_{\bz_0}^2 n_{\bs} n_{\bz_0}^{-2} = n_{\bs}$. 

If $G_{\bF, \ad}$ is of type $D_{2n+1}$, then $|\Omega_{\bal, \ad}|=4$. The element $\bz_0^3$ has no fixed elements in $\bBS$. For $i=2$, since $\bz_0^2$ fixes $\bs$, we have $\bs=\bs_k$ with $1<k<2n$. In this case, $\bz_0(\bs)=\bs_{2n+1-k}$. By Condition (2)(b)$^\dagger$ of \S \ref{sec:tits}, we have $n_{\bz_0} n_{\bs_k} n_{\bz_0}^{-1} = n_{\bs_{2n+1-k}}$ and $n_{\bz_0} n_{\bs_{2n+1-k}} n_{\bz_0}^{-1} = n_{\bs_k}$. So $n_{\bz_0}^2 n_{\bs} n_{\bz_0}^{-2} = n_{\bs}$.

\subsubsection{The remaining cases}
In this subsection, we assume that $\s^* \neq \s$. So in particular, we have $(\Omega_{\bal, \ad})_\s \neq 1$. 

We first discuss the case where the $\s$-action on $\Omega_{\bal, \ad}$ is trivial. 

If $G_{\bF, \ad}$ is of type $A_n$, then $g_\bz = n_{\bz_0}^i$ for some $i$ with $1 \le i<|(\Omega_{\bal, \ad})|$. We have $(\Ad({\bz_0}^{i|\CX|})(\bs)=\bs$. Since $\bz_0$ acts transitively on the gradients of elements of $\breve\Delta$, we see that $|(\Omega_{\bal, \ad})|$ divides $i |\CX|$. By Corollary \ref{cor:FTG}(2), we have $(\Ad(g_{\bz}) \circ \s)^{|\CX|}(n_\bs)=(\Ad(g_{\bz}))^{|\CX|}(n_\bs)=(\Ad(n_{\bz_0}^{i|\CX|})(n_\bs)=n_\bs$. 

If $G_{\bF, \ad}$ is not of type $A$ and the $\s$-action on $\Omega_{\bal, \ad}$ is trivial, then each $\s^*$-orbit on $\bBS$ is of size $1$ or the order $l$ of $\bz$ in $\bW_0$. The case where $\s^*(\bs)=\bs$ is handled in \S\ref{orbit1}. If $\s^*(\bs)\neq \bs$, then by Corollary \ref{cor:FTG}(1), $$(\s^*)^l(n_\bs)=\Ad(g_{\bz})^l \s^l(n_\bs)=\Ad(g_{\bz})^l(n_\bs)=n_\bs.$$

Next we discuss the case where the $\s$-action on $\Omega_{\bal, \ad}$ is non-trivial. Since $(\Omega_{\bal, \ad})_\s \neq 1$, $G_{\bF, \ad}$ is of type $A_{2n+1}$ or type $D_n$.

If $G_{\bF, \ad}$ is of type $A_{2n+1}$, then $g_{\bz}=n_{\by_{(1)}}$ and $\s(g_{\bz})=n_{\by_{(2n+1)}}$. We have $g_{\bz} \s(g_{\bz})=(\ba^\vee_1+\ba^\vee_3+\cdots+\ba^\vee_{2n+1})(-1) \in Z(\breve F)$. Note that any $\s^*$-orbit on $\bBS$ is of size $1$ or $2$. The case where $\s^*(\bs)=\bs$ is handled in \S\ref{orbit1}. If $\s^*(\bs)\neq \bs$, then $$(\s^*)^2(n_\bs)=\Ad(g_{\bz} \s(g_{\bz})) \s^2(n_\bs)=\Ad((\ba^\vee_1+\ba^\vee_3+\cdots+\ba^\vee_{2n+1})(-1)) n_\bs=n_\bs.$$

If $G_{\bF, \ad}$ is of type $D_n$ with $n$ odd, then $g_{\bz}=n_{\by_{(n)}}$ and $\s(g_{\bz})=n_{\by_{(n-1)}}$. By Corollary \ref{cor:FTG}(3), $g_{\bz} \s(g_{\bz})=1$ or $(\ba^\vee_{n-1}+\ba^\vee_n)(-1)$. In either case, $g_{\bz} \s(g_{\bz}) \in Z(\breve F)$. We then follow the same argument as the type $A_{2n+1}$ case above. 

If $G_{\bF, \ad}$ is of type $D_n$ with $n$ even and $g_{\bz}=n_{\by_{(1)}}$, then  $\s(g_{\bz})=n_{\by_{(1)}}$. By Corollary \ref{cor:FTG}(4), $g_{\bz} \s(g_{\bz})=(\ba^\vee_{n-1}+\ba^\vee_n)(-1) \in Z(\breve F)$. We then follow the same argument as the type $A_{2n+1}$ case above. 

If $G_{\bF, \ad}$ is of type $D_n$ with $n$ even and $g_{\bz}=n_{\by_{(n-1)}}$ or $n_{\by_{(n)}}$, then by Corollary \ref{cor:FTG}(4), $g_{\bz} \s(g_{\bz})=z n_{\by_{(1)}}$ for some $z \in Z(\breve F)$. Note that each $\s^*$-orbit on $\bBS$ is of size $1$ or $4$. We have \begin{align*}(\Ad(g_\bz) \circ \s)^4 &=\Ad(g_{\bz} \s(g_{\bz})) \Ad(\s^2(g_{\bz} \s(g_{\bz})) \circ \s^4 \\ &=\Ad(z n_{\by_{(1)}} \s^2(z n_{\by_{(1)}})) \circ \s^4 \\ &=\Ad(z \s^2(z) n_{\by_{(1)}}^2) \circ \s^4=\s^4.\end{align*}  Thus $$(\s^*)^4(n_\bs)=(\Ad(g_\bz) \circ \s)^4(n_\bs)=\s^4(n_\bs)=n_\bs.$$

This finishes the verification of \eqref{SCX} in all the remaining cases and thus finishes the proof of Proposition \ref{prop:power-s}. 

\subsection{The $\s^*$-fixed liftings of $\Omega_{\bal}^{\s^*}$}\label{constaubs}
 
Let $\breve \t=t_\blambda \by \in \Omega_{\bal}^{\s^*}$. We will set $m(\btau) = n_\blambda m(\by)$ for a suitable $m(\by) \in \breve\CT$.  
 
Let us first choose $m(\by)$ when $G_{\bF,\ad}$  is $\bF$-simple. If $G_{\bF, \ad}$ is of type $D_n$ with $n$ even, then we set $m(\by)=n_{\by}$. Otherwise, $\Omega_{\bal, \ad}$ is a cyclic group and $\by=\bzg^j$ for $0 \le j<|\Omega_{\bal, \ad}|$. Set $m(\by)=n_{\bzg}^j$ and $m(\breve \t)=n_\blambda  m(\by)$. 
 
If $G_{\bF, \ad}$ is not $\bF$-simple, then 
 \begin{align}\label{WGS}
     W(G, T)=W(G, T)^{(1)} \times \cdots \times W(G, T)\ik,
 \end{align} where $W(G, T)^{(1)} \cong \cdots \cong W(G, T)\ik$ are irreducible finite Weyl groups and the action of $\s$ permutes transitively the irreducible factors $W(G, T)^{(1)}, \ldots, W(G, T)\ik$. There exist $\by^{(1)} \in W(G, T)^{(1)}$ such that $\by=\by^{(1)} \s(\by^{(1)}) \cdots \s^{k-1}(\by^{(1)})$. Define $m(\by^{(1)})$ as in the preceding paragraph. More precisely, if $G_{\bF, \ad}^{(1)}$ is of type $D_n$ with $n$ even, then we set $m(\by^{(1)})=n_{\by^{(1)}}$. Otherwise, $\Omega_{\bal, \ad}^{(1)}$ is a cyclic group and $\by^{(1)}=\bzg^j$ for $0 \le j<|\Omega_{\bal, \ad}^{(1)}|$. Set $m(\by^{(1)})=n_{\bzg}^j$. Let \[m(\by) = m(\by^{(1)}) \sigma(m(\by^{(1)})) \cdots \sigma^{k-1}(m(\by^{(1)})).\]
 
The main result of this subsection is the following.

\begin{proposition}\label{prop:m-t}
Let $\breve \t \in \Omega_{\bal}^{\s^*}$. Then $\s^*(m(\breve \t))=m(\breve \t)$. 
\end{proposition} 

\subsubsection{Reduction step} We begin with a simple lemma.
\begin{lemma}\label{lem:sigmatau}
For each $\btau \in \Omega_{\bal}$, we have $\sigma^*(\btau) = \sigma(\btau)$.
\end{lemma}
\begin{proof}
Let $\btau=t_{\blambda} \by$. Then $\sigma^*(\btau) = t_{\sigma^*(\blambda)}t_{\bea - y'(\bea)}\by'$, where $\by'=Ad(\bz)(\sigma(\by))$. Since $\Omega_{\bal, \ad}$ is abelian, $\by' = \sigma(\by)$ and $\sigma^*(\blambda) -\sigma(\blambda) = \bz(\sigma(\blambda))-\sigma(\blambda)=\bea - y'(\bea)+\bmu$ for a suitable $\bmu \in X_*(Z^0)$. In particular, we have $\bz$ commutes with $\by'=\s(\by)$. 

Let $l$ be the order of $\bz$. Then $\bz^{i+1}(\sigma(\blambda)) -\bz^{i}(\sigma(\blambda)) = \bz^{i}(\bea) - y' \bz^{i}(\bea)+ \bmu$ for any $i$. Since $\nua^l=1$, we have $\bea+ \bz(\bea) + \cdots+\bz^{l-1}(\bea) =0$. Thus $$l \bmu = \sum_{i=0}^{l-1} \bigl(\bz^{i}(\bea) - y' \bz^{i}(\bea)+ \bmu \bigr)=\sum_{i=0}^{l-1} \bigl(\bz^{i+1}(\sigma(\blambda)) -\bz^{i}(\sigma(\blambda))\bigr)=0$$ and hence $\bmu=0$. So $\sigma^*(\btau) = t_{\sigma(\blambda)} \sigma(\by) = \sigma(\btau)$.  
\end{proof}

Let $\breve \t=t_\blambda\by \in \Omega_{\bal}^{\s^*}=\Omega_{\bal}^{\s}$. Then we have $\s(\by)=\by$ and  $\blambda=\s^*(\blambda)+\bea-\Ad(\bz)(\sigma(\by))(\bea)$. 

By \eqref{lem1:trans} and Lemma \ref{lem2:trans}, \begin{align*} \s^*(m(\breve \t)) &=\s^*(n_\blambda m(\by))=\s^*(n_{\blambda}) \s^*(m(\by)) \\ &=n_{\s^*(\blambda)+\bea-\Ad(\bz)(\sigma(\by))\bea} \Ad(g_{\bz}) \s(m(\by)) \\ &=n_{\blambda} \Ad(g_{\bz}) \s(m(\by)). \end{align*}

To verify $\s^*(m(\breve \t))=m(\breve \t)$, it remains to show \begin{enumerate}
    \item $\s(m(\by))=m(\by)$; 
    \item $\Ad(g_{\bz})(m(\by)) = m(\by)$. 
\end{enumerate}

Now we show that it suffices to check the case where $G_{\bF, \ad}$ is $\bF$-simple. 

\begin{lemma} We have \begin{enumerate}
    \item $\s(m(\by))=m(\by)$  if and only if $\s^k(m(\by^{(1)})) = m(\by^{(1)})$. 
    \item $\Ad(g_{\bz})(m(\by)) = m(\by)$ if and only if $\Ad(g_{\bz})(m(\by^{(1)})) = m(\by^{(1)})$.
\end{enumerate}

\end{lemma}
\begin{proof}
Note that $\sigma(\by) =\by$ if and only if $\sigma^k(\by^{(1)}) = \by^{(1)}$. By the definition of $m(\by)$, we have $\sigma(m(\by)) = \sigma(m(\by^{(1)})) \sigma^2(m(\by^{(1)})) \cdots \sigma^{k}(m(\by^{(1)}))$. Further, \eqref{WGS} implies that
\begin{align}\label{Tfin}
    \breve\CT_{\fin} \cong \breve\CT_{\fin}^{(1)} \times \breve\CT_{\fin}^{(2)} \times \cdots \times \breve\CT_{\fin}^{(k)},
\end{align}
where $ \breve\CT_{\fin}^{(i)}$ is the finite Tits group attached to $W(G,T)^{(i)}$. Hence 
\[\sigma(m(\by)) = \sigma^{k}(m(\by^{(1)})) \sigma(m(\by^{(1)})) \sigma^2(m(\by^{(1)})) \cdots \sigma^{k-1}(m(\by^{(1)})).\] 
Now it follows that $\sigma(m(\by) = m(\by)$ if and only if $ \sigma^{k}(m(\by^{(1)}))= m(\by^{(1)})$.

For (2), since $g_\bz \in \breve\CT_{\fin}^{(1)}$, and $\sigma^i(m(\by^{(1)}) \in \breve\CT_{\fin}^{(i)}$ for each $0 \le i \le k-1$, we see using \eqref{Tfin} that $g_\bz$ commutes with $\sigma^i(m(\by^{(i)})$ for all $i \geq 1$. Hence $\Ad(g_{\bz})(m(\by)) = m(\by)$ if and only if $\Ad(g_{\bz})(m(\by^{(1)})) = m(\by^{(1)})$.
\end{proof}

\subsubsection{Proof of Proposition \ref{prop:m-t} for $\breve F$-simple groups}

We assume that $G_{\bF, \ad}$ is $\bF$-simple and we drop the subscripts in the discussion below.  In particular, we may assume $\Omega_{\bal, \ad}$, $\nuag$ are as in Table \ref{tab:tableOmega}. So $g_\bz = n_\bz$ if $\Omega_{\bal, \ad}$ is of type $D_n$ with $n$  even. Otherwise, $g_\bz = n_{\bz_0}^i$ for a suitable $0 \le i <|(\Omega_{\bal, \ad})_{\sigma}|$. Also $m(\by) = n_\by$ if $\Omega_{\bal, ad}$ is of type $D_n$ with $n$ even. Otherwise, $m(\by) = n_{\bz_0}^j$ for a suitable $0 \le j <|(\Omega_{\bal, \ad}|$.

We show that $\Ad(g_\bz)(m(\by)) = m(\by)$. 

When $G_{\bF, \ad}$ is of type $D_n, n$ even, this is a consequence of Corollary \ref{cor:FTG}(4). Otherwise, $g_{\bz}$ and $m(\by)$ are both powers of $n_{\bz_0}$ and the claim is obvious.

We show that $\sigma(m(\by)) = m(\by)$. 

The proof involves a detailed case-by-case analysis. Recall that the representatives $\{n_{s_\ba}\;|\; \ba \in \breve\Delta\}$ satisfies $H(\breve\Delta, \sigma)$. When $G_{\bF, \ad}$ is of type $D_n, n$ even, we have $\s(m(\by))=\s(n_\by)=n_\by=m(\by)$. Next we consider the case where $\Omega_{\bal, \ad}$ is cyclic. If the action of $\sigma$ on $\Omega_{\bal,ad}$ is trivial, then $\sigma(\bz_0) = \bz_0$ and $\sigma(n_{\bz_0}) = n_{\bz_0}$. In this case, $\sigma(m(\by)) = \sigma(n_{\bz_0}^j) = \sigma(n_{\bz_0})^j = n_{\bz_0}^j = m(\by)$. It remains to prove the claim when $\Omega_{\bal, \ad}$ is cyclic and the action of $\sigma$ on $\Omega_{\bal, \ad}$ is non-trivial. 

Recall that $j: G_\bF \rightarrow G_{\bF, \ad}$ is the adjoint quotient. Let $\btau \in \Omega_{\bal}^{\s^*}$. If $j(\btau)=1$, then $\by=1$ and $m(\by)=1$. In this case, $\s(m(\by))=1=m(\by)$. 

Now we assume that $j(\btau) \neq 1$. This happens when $G_{\bF, \ad}$ is of type $A_{2n+1}$ and $j=n+1$ or $G_{\bF, \ad}$ is of type $D_{2n+1}$ and $j=2$. In both these cases $\sigma(\nuag) = \nuag\i$ and $\sigma(\bz_0) = \bz_0^{-1}$. 

If $G_{\bF, \ad}$ is of type $A_{2n+1}$, then $m(\by)=n_{\bz_0}^{n+1}$ and $\bz_0^{n+1}=\by_{(n+1)}$. By Corollary \ref{cor:FTG}(2), we have     
\[n_{\bz_0}^{n+1} = \begin{cases} n_{\bz_0^{n+1}}, & \text{ if } n \text{ is even},\\
    (a_1 ^\vee + a_3^{\vee} + \cdots + a_{n}^\vee)(-1) n_{\bz_0^{n+1}}, & \text{ if } n \text{ is odd}.
    \end{cases} 
\]

We have $\s(\by_{(n+1)})=\by_{(n+1)}$. Thus 
\[\s(n_{\bz_0}^{n+1}) = \begin{cases} n_{\bz_0}^{n+1}, & \text{ if } n \text{ is even},\\
    (a_1 ^\vee + a_3^{\vee} + \cdots + a_{2n+1}^\vee)(-1) n_{\bz_0}^{n+1}, & \text{ if } n \text{ is odd}.
    \end{cases} 
\]

We identify $\Omega_{\bal, \ad}$ with $X_*(T_{\ad})/X_*(T_{sc})$. Under this identification, 
\begin{align*}
    j(\btau) &= (n+1)\nuag=(n+1) \frac{\ba_1^\vee +2\ba_2^\vee +3 \ba_3^\vee +\cdots+(2n+1)\ba_{2n+1}^\vee}{2n+2} \\ &=\frac{\ba_1^\vee  +3 \ba_3^\vee + 5\ba_5^\vee +\cdots+ (2n+1)\ba_{2n+1}^\vee}{2} +  \ba_2^\vee +2 \ba_4^\vee +\cdots+n \ba_{2n}^\vee \\ & \equiv \frac{\ba_1^\vee  +\ba_3^\vee + \ba_5^\vee +\cdots+\ba_{2n+1}^\vee}{2}\mod X_*(T_{sc}).
\end{align*}

Since $\btau \in X_*(T)/X_*(T_{sc})$ and $j$ acts as identity on $X_*(T_{sc})$, $\frac{\ba_1^\vee  +\ba_3^\vee + \ba_5^\vee +\cdots+\ba_{2n+1}^\vee}{2} \in X_*(T)$. Hence 
\begin{align*}
    (\ba_1^\vee +\ba_3^\vee +\cdots+\ba_{2n+1}^\vee)(-1)= \left(\left(\frac{\ba_1^\vee  +\ba_3^\vee + \ba_5^\vee +\cdots+\ba_{2n+1}^\vee}{2}\right)(-1)\right)^2=1 \in G(\breve F).
\end{align*}

Therefore we have $\s(m(\by))=\s(n_{\bz_0}^{n+1})=n_{\bz_0}^{n+1}=m(\by)$. 

If $G_{\bF, \ad}$ is of type $D_{2n+1}$, we have $m(\by) = n_{\bz_0}^2$, and by Corollary \ref{cor:FTG}(3), $\s(m(\by)) = tm(\by)$ where $t = (\ba_{2n}^\vee + \ba_{2n+1}^\vee)(-1)$. We claim that $t=1$ in $G(\bF)$. The argument is similar to type $A_{2n+1}$. Consider the element $j(\btau)\in \Omega_{\bal, \ad} =X_*(T_{\ad})/X_*(T_{sc})$. Then \[\nuag = \frac{\ba_1^\vee +2\ba_2^\vee +\cdots+ (2n-1) \ba_{2n-1}^\vee + \frac{1}{2} (2n-1) \ba_{2n}^\vee +\frac{1}{2}(2n+1)\ba_{2n+1}^\vee}{2} \mod X_*(T_{sc}).\] 
Then $j(\btau) \equiv 2\nuag \equiv \frac{1}{2} \ba_{2n}^\vee +\frac{1}{2}\ba_{2n+1}^\vee \mod X_*(T_{sc})$.
Since $2\btau \in X_*(T_{sc})$ and $j$ acts as identity on $X_*(T_{sc})$, we see that
\begin{align*}
    \btau &\equiv\frac{1}{2} \ba_{2n}^\vee +\frac{1}{2}\ba_{2n+1}^\vee \mod X_*(T_{sc}). 
\end{align*}
Since $\btau \in X_*(T)/X_*(T_{sc})$, we see that $ \frac{1}{2} \ba_{2n}^\vee +\frac{1}{2}\ba_{2n+1}^\vee \in X_*(T)$. Hence 
\begin{align*}
    t &= (\ba_{2n}^\vee + \ba_{2n+1}^\vee)(-1)= \left(\left(\frac{1}{2} \ba_{2n}^\vee +\frac{1}{2}\ba_{2n+1}^\vee\right)(-1)\right)^2=1.
\end{align*}
Hence $\s(m(\by))=m(\by)$. 

This finishes the proof of Proposition \ref{prop:m-t}. 

\subsection{Proof of Theorem \ref{thm:tits-F}} 
For $\bs \in \bBS$, let $m(\bs) \in \breve \CT$ be the lifting of $\bs$ in Corollary \ref{cor:m(s)}. For any $\breve \t \in \Omega_{\bal}^{\s^*}$, let $m(\breve \t) \in \breve \CT$ be the lifting of $\breve \t$ constructed in \S \ref{constaubs}. Then $\s^*(m(\breve \t))=m(\breve \t)$ for all $\breve \t \in \Omega_{\bal}^{\s^*}$ and $\s^*(m(\bs))=m(\s^*(\bs))$ for all $\bs \in \bBS$. 

We set $\CT=\breve \CT^{\s^*}$.

For any $s \in \BS$, we have $s=\bw_{\CX}$ for some $\s$-orbit $\CX$ in $\breve \Delta$ with $\bW_{\CX}$ finite (see \S \ref{sec:WF}). Let $\bw_{\CX}=\bs_{i_1} \cdots \bs_{i_n}$ be a reduced expression of $\bw_{\CX}$ in $\bW_{\af}$. Then $\bw_{\CX}=\s(\bs_{i_1}) \cdots \s(\bs_{i_n})$ is again a reduced expression of $\bw_{\CX}$ in $\bW_{\af}$. We have \begin{align*} m(\bw_{\CX}) &=m(\bs_{i_1}) \cdots m(\bs_{i_n})=m(\s(\bs_{i_1})) \cdots m(\s(\bs_{i_n}))=\s(m(\bs_{i_1})) \cdots \s(m(\bs_{i_n})) \\ &=\s(m(\bw_{\CX})).
\end{align*}

In particular, $m(s)=m(\bw_{\CX}) \in \CT=\breve \CT^{\s^*}$. 

Let $w \in W_{\af}$ and $s_{i_1} \cdots s_{i_n}$ be a reduced expression of $w$ in $W$. We set $m(w)=m(s_{i_1}) \cdots m(s_{i_n})$. Then $m(w) \in \CT$. Suppose that $s'_{i_1} \cdots s'_{i_n}$ is another reduced expression of $w$ in $W$. By \S \ref{sec:WF} (a), $\breve \ell(w)=\breve \ell(s_{i_1})+\cdots+\breve \ell(s_{i_n})=\breve \ell(s'_{i_1})+\cdots+\breve \ell(s'_{i_n})$. Since $\{m(\bs)\;|\; \bs \in \bBS\}$ satisfies the Coxeter relations, by condition (2)(b) $^\dagger$ in \S \ref{sec:tits}, $m(s_{i_1}) \cdots m(s_{i_n})=m(s'_{i_1}) \cdots m(s'_{i_n})$. In other words, $m(w)$ is independent of the choice of reduced expression in $W$. Finally for $w \in W$, we have $w=w_1 \t$ for a unique $w_1 \in W_{\af}$ and $\t \in \Omega_{\al}=\Omega_{\bal}^{\s^*}$. We set $m(w)=m(w_1) m(\t)$. Then $m(w) \in \CT$. In other words, the map $\phi: \CT \to W$ is surjective. We have $$\ker(\phi)=\ker(\breve \phi) \cap \CT_{\af}=\breve S_2 \cap \CT_{\af}=S_2.$$ 

It remains to show that for each $a \in \Delta$, we have $m(s_a)^2 = b^\vee(-1) $ where $b$ is the gradient of $a$. By \cite[\S 5.1]{BT2}, we know that the elements of $\Delta$ and in bijection with $\sigma^*$-orbits $\CX$ of $\breve\Delta$ with  $\ba|_{\CA(A,F)}$ non-constant for $\ba \in \CX$. Let $\CF$ denote the $\BR$-vector space of affine linear functions on $\CA(A,F)$. Then $\CF$ may be identified  with the $\sigma^*$-invariants of $\breve\CF$ where $\breve\CF$ is the $\BR$-vector space of affine functions on $\CA(T,\bF)$ (which we have identified with $V = X_*(T) \otimes \BR$ after choosing a special point). Under this identification, we have for $a \in \Delta$,  \[ a \mapsto \frac{1}{|\CX|}\sum_{\ba \in \CX}  \ba,\] where $\CX$ is the $\sigma^*$-orbit on $\breve\Delta$ that corresponds to $a$.  

Fix a $\sigma^*$-invariant scalar product $\langle \cdot, \cdot \rangle$ on $V = X_*(T) \otimes \BR$ and we identify $V$ with $V^*$ via this inner product. We may extend this to a scalar product on $\breve\CF$ by setting $\langle \breve f, \breve g \rangle = \langle D\breve f, D\breve g \rangle$, where $D\breve f \in V$ is the gradient of $\breve f$. For any $\breve f\in \breve \CF$ with $D\breve f \neq 0$, let $\breve f^\vee = \frac{2\breve f}{\langle \breve f, \breve f \rangle}$. 

Then for $a \in \Delta$, we have $a^\vee = \frac{2a}{\langle a, a\rangle}$. Fix $\ba \in \breve \Delta$ whose restriction to $\CA(A, F)$ is the affine root $a$. Then 
\[\langle a, a \rangle = \frac{1}{|\CX|} \sum_{\ba' \in \CX} \langle \ba, \ba'\rangle.\]
This implies that
\begin{equation}\label{cadefn}
    a^\vee = c_a \sum_{\ba' \in \CX} \ba'^\vee, \;\; b^\vee = c_a \sum_{\ba' \in \CX} \bb'^\vee
\end{equation} where \[c_a = \frac{\langle \ba, \ba\rangle }{\sum_{\ba' \in \CX} \langle \ba, \ba'\rangle} = \frac{\langle \bb, \bb\rangle }{\sum_{\ba' \in \CX} \langle \bb, \bb'\rangle}.\] 

Now, let us prove that $m(s_a)^2 = b^\vee(-1)$.  We may easily reduce ourselves to the case where $G_\bF$ is $\bF$-simple and simply connected. Via a simple case-by-case analysis, each $\s^*$-orbit $\CX$ consists of simple roots in $\breve \D$ whose corresponding Dynkin diagram is either a product of $A_1$'s or is a single copy of $A_2$. 

In the former case, $c_a=1$ and by \eqref{cadefn},\[ m(s_a)^2 = m(s_{\ba_1})^2 m(s_{\ba_2}^2) \cdots m(s_{\ba_k})^2 = \bb_1^\vee(-1)\bb_2^\vee(-1)\cdots \bb_k^\vee(-1) = b^\vee(-1),\]
where $\CX = \{\ba_1, \ba_2, \cdots,  \ba_k\}$. 

In the latter case, $\CX = \{\ba_1, \ba_2\}$ and $\ba_1+\ba_2$ is an affine root. In this case, $c_a=2$. By \eqref{cadefn}, \[m(s_a)^2 = (m(s_{\ba_1})m(s_{\ba_2})m(s_\ba))^2 = 1 = b^\vee(-1).\]

Thus $\CT$ is a Tits group of $W$ and $\{m(w)\;|\; w \in W\}$ is a Tits cross-section of $W$ in $\CT$.

\end{document}